\documentclass[11pt]{article}
\usepackage[english]{babel}
\usepackage{a4wide}
\usepackage{amssymb, amsmath}
\usepackage{color}
\usepackage{enumitem}
\usepackage{amsthm}
\usepackage{wasysym}
\usepackage{stmaryrd}
\usepackage{tikz-cd}
\usetikzlibrary{arrows}
\usepackage{hyperref}
\usepackage[capitalise]{cleveref}

\newcommand{\ruitje}{\hfill $\lozenge$}
\newcommand{\pf}{\rightharpoonup}
\newcommand{\To}{\Rightarrow}
\newcommand{\denotes}{\!\downarrow}
\newcommand{\s}{\hspace{1pt}}
\newcommand{\mono}{\hookrightarrow}

\newcommand{\lval}{\llbracket}
\newcommand{\rval}{\rrbracket}

\newcommand{\id}{\mathop{\mathgroup\symoperators id}\nolimits}
\newcommand{\asm}{\mathsf{Asm}}
\newtheorem{thm}{Theorem}[section]
\newtheorem{lem}[thm]{Lemma}
\newtheorem{cor}[thm]{Corollary}
\newtheorem{prop}[thm]{Proposition}
\theoremstyle{definition}
\newtheorem{defi}[thm]{Definition}
\newtheorem{rem}[thm]{Remark}
\newtheorem{ex}[thm]{Example}

\newtheorem{conv}[thm]{Convention}

\author{Jetze Zoethout\\
Department of Mathematics, Utrecht University}

\date{October 18, 2019}

\title{Internal Partial Combinatory Algebras and their Slices}

\begin{document}

\maketitle

\centerline{\textsc{Abstract}}

\begin{footnotesize}{\setlength{\parindent}{0pt}
A partial combinatory algebra (PCA) is a set equipped with a partial binary operation that models a notion of computability. This paper studies a generalization of PCAs, introduced by W.\@ Stekelenburg (\cite{wps}), where a PCA is not a set but an object in a given regular category. The corresponding class of categories of assemblies is closed both under taking small products and under slicing, which is to be contrasted with the situation for ordinary PCAs. We describe these two constructions explicitly at the level of PCAs, allowing us to compute a number of examples of products and slices of PCAs. Moreover, we show how PCAs can be transported along regular functors, enabling us to compare PCAs constructed over different base categories. Via a Grothendieck construction, this leads to a (2-)category whose objects are PCAs and whose arrows are generalized applicative morphisms. This category has small products, which correspond to the small products of categories of assemblies, and it has finite coproducts in a weak sense. Finally, we give a criterion when a functor between categories of assemblies that is induced by an applicative morphism has a right adjoint, by generalizing the notion of computational density from \cite{hofstrajaap}.
}\end{footnotesize}

\section{Introduction}

A partial combinatory algebra (PCA) is an abstract model of computation that generalizes the classical notion of computability on the set of natural numbers. A more precise definition will be given below. These models can be studied from the point of view of category theory. Every PCA $A$ gives rise to a category of assemblies $\asm(A)$, which may be viewed as the category of all data types that can be implemented in $A$. Moreover, the ex/reg completion of $\asm(A)$ is always a topos, called the realizability topos of $A$ and denoted by $\mathsf{RT}(A)$. In this topos, the internal logic is governed by computability in the model $A$.

The fundamental theorem of topos theory states that a slice category of a topos is again a topos. This implies that a category of the form $\mathsf{RT}(A)/I$ is also a topos. However, it is not in general a realizability topos, which we can show as follows. In every realizability topos, the terminal object is projective. In the slice topos $\mathsf{RT}(A)/I$, on the other hand, the terminal object is projective if and only if $I$ itself is projective in $\mathsf{RT}(A)$. Therefore, if we let $I$ be a \emph{non-projective} object of $\mathsf{RT}(A)$, which almost always exists, then $\mathsf{RT}(A)/I$ will not be a realizability topos. Similar observations apply for categories of assemblies. Explicitly, $\asm(A)$ is always a quasitopos, and quasitoposes are closed under slicing. However, a slice of $\asm(A)$ is not in general a category of assemblies, for the same reason we presented above.

This leads to the question whether there is a natural class of categories that contains all categories of the form $\asm(A)$ and \emph{is} closed under slicing. Recently, J.~Frey has given an extensional characterization of toposes of the form $\mathsf{RT}(A)$, where $A$ is a PCA (\cite{frey2}, Theorem 4.6). This characterization provides an important hint as to where to look for such a class. But first, let us give a more precise definition of a PCA.
\begin{defi}\label{PCA_intro}
A \emph{partial combinatory algebra} (PCA) is a nonempty set $A$ equipped with a \emph{partial} binary map $A\times A\pf A\colon (a,b)\mapsto ab$, called \emph{application}, for which there exist $\mathsf{k},\mathsf{s}\in A$ such that
\begin{itemize}
\item[(i)]	$(\mathsf{k}a)b$ is defined and equal to $a$;
\item[(ii)]	$(\mathsf{s}a)b$ is defined;
\item[(iii)]	if $(ac)(bc)$ is defined, then $((\mathsf{s}a)b)c$ is defined and equal to $(ac)(bc)$,
\end{itemize}
for all $a,b,c\in A$. \ruitje
\end{defi}
\begin{rem}
Other sources employ a slightly stronger definition of PCA, where in item (iii), $((\mathsf{s}a)b)c$ should be defined \emph{exactly} when $(ac)(bc)$ is defined. A PCA in our sense is then called a \emph{weak} PCA. It has been shown (\cite{weakPCA}, Theorem 5.1) that there is no essential difference between these two notions. We choose the above as our definition of a PCA because for our purposes, it is more pleasant to work with. \ruitje
\end{rem}
It follows from this definition that $A$ satisfies an abstract version of the $Smn$-theorem: every expression built using the application map can be computed using an element from $A$ itself. We refer to \cref{first_section} for a more precise formulation of `expression' and what it means to compute such an expression (Definitions 2.2 and 2.3).

A well-known generalization of a PCA is that of a \emph{relative} PCA. A relative PCA is a pair $(A,C)$, where $A$ is a PCA and $C$ is a subset of $A$ that is closed under the application map from $A$, and such that suitable elements $\mathsf{k}$ and $\mathsf{s}$ as in \cref{PCA_intro} may be found in $C$. We regard the elements of $C$ as `computable' elements that may act on possibly non-computable data. A certain operation then counts as computable if it is computed by some element from $C$. The constructions $\asm$ and $\mathsf{RT}$ mentioned above can be generalized to these relative PCAs.

A crucial ingredient in Frey's characterization is the fact that toposes of the form $\mathsf{RT}(A)$ carry a geometric inclusion $\mathsf{Set}\mono\mathsf{RT}(A)$, where the inverse image is the global sections functor. Similarly, categories of the form $\asm(A)$ allow an adjunction with $\mathsf{Set}$, where the left adjoint is the global sections functor. Slicing over an assembly $I$ affects this adjunction in two important ways. First of all, we get an adjunction with a \emph{slice} $\mathsf{Set}/|I|$ of $\mathsf{Set}$, rather than $\mathsf{Set}$ itself. This suggests that we should allow for other `base categories' than $\mathsf{Set}$. Second, the left adjoint $\asm(A)/I\to \mathsf{Set}/|I|$ ceases to be the global sections functor. This is no surprise, since the codomain of this left adjoint is no longer $\mathsf{Set}$. But we can even say something stronger: the left adjoint does not even commute with the global sections functors $\asm(A)/I\to\mathsf{Set}$ and $\mathsf{Set}/|I|\to\mathsf{Set}$. This situation, where the left adjoint of the adjunction is not the global sections functor, is typical of relative realizability (see also \cite{frey}, Corollary 4.11.7(i)). This suggests that we should allow for some kind of \emph{relative} PCA.

A class of PCA-like structures satisfying both these desiderata was described by W.~Stekelenburg in his PhD thesis \cite{wps}. Stekelenburg considers an even more general notion of relative realizability than the one described above, where certain (nonempty) subsets of $A$ are declared to be `realizing sets'. In this setting, an operation counts as computable if there is a realizing set, all whose elements compute this operation. This more general notion is crucial for removing the difficulty surrounding the projective terminal object mentioned above, as we will explain in \cref{section_slice} below.

We explain the relevant notions from \cite{wps} in \cref{first_section} below, where we shall simply use `PCA' to refer to the generalized notion of a PCA. We shall use `classical (relative) PCA' to refer to the notion defined in \cref{PCA_intro}. Stekelenburg also defined a notion of morphism between PCAs over a given base category, which generalizes a notion formulated by Longley (\cite{longleyphd}). In \cref{section_appmor}, we define such a notion for PCAs over \emph{different} base categories, making PCAs the objects of a 2-category. We investigate the interaction of this 2-category with the construction $\asm$ in \cref{section_asm}. Then, in \cref{section_slice}, we present an explicit description of the slice of a category of assemblies, and we use this description to calculate a number of examples of slices. Finally, we discuss the notion of computational density (see \cite{hofstrajaap}) in the present setup, in \cref{section_cd}.

Since we will be working with 2-categories, a few comments on terminology are in order. In general, the prefix `2-' will signify that we discuss a notion enriched over categories. Thus, a 2-category has a strictly associative and unital composition of 1-cells, and a 2-functor strictly preserves the identity 1-cells and the composition of 1-cells. A 2-(co)product is a (co)product whose universal property is expressed by an isomorphism of categories. We use the term `pseudo(co)product', on the other hand, for a (co)product with a universal property expressed by an equivalence of categories, rather than an isomorphism.

We also mention that, even though we replace the category of sets by a more general category, we still presuppose some ambient set theory to work in. In particular, we suppose we have a notion of `small'.

Finally, I wish to thank my PhD supervisor Jaap van Oosten, with whom I have had many constructive conversations on the topics discussed here, and who has provided countless valuable comments on earlier draft versions of this paper.

\section{PCAs and Assemblies}\label{first_section}

In this section, we introduce our main object of study: relative partial combinatory algebras constructed over a regular category. Many of the definitions provided in this section can also be found in some form in \cite{wps}. We deviate in one important respect from \cite{wps}, because all our base categories will be regular categories, rather than Heyting categories. This also requires adjusting certain definitions from \cite{wps} so that they work for this more general context.
 
Throughout this section, we will work with a fixed locally small regular category $\mathcal{C}$. Such a category soundly interprets (typed) regular logic, that is, the logic of $=$, $\top$, $\wedge$ and $\exists$. If $\varphi(x_1,\ldots, x_n)$ is a regular formula and $x_i$ is of type $X_i$, then we denote its interpretation in $\mathcal{C}$ by \[\{(x_1, \ldots, x_n)\in X_1\times\cdots\times X_n\mid \varphi(x_1, \ldots, x_n)\}\subseteq X_1\times\cdots\times X_n.\]
(Here we do not require that all the variables $x_1,\ldots,x_n$ actually occur free in $\varphi$.) A regular sequent is an expression of the form $\varphi\vdash_\Gamma\psi$, where $\Gamma$ is a context of typed variables and $\varphi$ and $\psi$ are regular formulas whose free variables are among $\Gamma$. If $\Gamma = x_1, \ldots, x_n$ and $x_i$ is of type $X_i$, then such a sequent is valid in $\mathcal{C}$ if
\[\{(x_1, \ldots, x_n)\in X_1\times\cdots\times X_n\mid \varphi\}\subseteq \{(x_1, \ldots, x_n)\in X_1\times\cdots\times X_n\mid \psi\}.\]
In this case, we write $\varphi\models_\Gamma\psi$, or $\mathcal{C}:\varphi\models_\Gamma\psi$ if we need to clarify in which category we are working.


As is customary when working with an internal logic, we will freely use subobjects and arrows of $\mathcal{C}$ as relation resp.\@ function symbols of our language. We will frequently, and usually implicitly, employ the soundness of the interpretation to derive that certain regular sequents are valid in $\mathcal{C}$ given that others are. If we give such a soundness argument explicitly, we will signal this to the reader by writing `reason inside $\mathcal{C}$'.

We start by defining the suitable generalization of a partial applicative structure.
\begin{defi}
A \emph{partial applicative structure} over $\mathcal{C}$ (PAS) is an inhabited object $A$ of $\mathcal{C}$ equipped with a partial binary map $A\times A\pf A$, called \emph{application}. Explicitly, the application map is given by a subobject $D\subseteq A\times A$ and an arrow $D\to A\colon (a,b)\mapsto a\cdot b$. \ruitje
\end{defi}

We write $a\cdot b\denotes$ for the formula $D(a,b)$. When no confusion can arise, we will just write $ab$ instead of $a\cdot b$. In general, application maps will not be associative. In order to avoid an unmanageable number of brackets, we adopt the convention the application associates to the left, that is, we write $abc$ as an abbreviation of $(ab)c$.

In order to define combinatorial completeness for PASs, we first need to introduce terms. Suppose that a countable stock of distinct variables is given; we shall use $x,y,z$ to range over variables.

\begin{defi}
The set of \emph{terms} is defined recursively by:
\begin{itemize}
\item[(i)]	every variable is a term;
\item[(ii)]	if $s$ and $t$ are terms, then $(s\cdot t)$ is a term as well. \ruitje
\end{itemize}
\end{defi}
The conventions for application also apply to terms. That is, we omit $\cdot$ and brackets whenever possible, subject to the stipulation that $rst$ abbreviates $(rs)t$. Every term $t = t(x_0, \ldots, x_n)$, where $n\geq 0$, determines a partial map $A^{n+1}\pf A$ in the obvious way, and we denote this map by $\lambda\vec{x}. t$. The domain of $\lambda\vec{x}.t$ can be expressed by a regular formula involving $D$ and the application map. We abbreviate this formula by $t(\vec{a})\denotes$, where $\vec{a}:A^{n+1}$. For example, $abc\denotes$ may be expressed as $D(a,b)\wedge D(ab,c)$. One may object that the function symbol for application is a unary function symbol with domain $D$, rather than a binary function symbol taking inputs from $A$. We can circumvent this difficulty by expressing the application map by a tertiary single-valued relation symbol on $A$, expressing `$ab=c$'. The formula $abc\denotes$ may then be rendered as $\exists w:A\s (ab=w\wedge D(w,c))$. Likewise, if $t(\vec{a})\!\denotes$, then we will freely use the expression $t(\vec{a})$ in our formulas. E.g., we may write $abc\denotes\wedge\ \varphi(abc)$, which should really be read as $\exists v,w:A\s (ab=w\wedge wc=v\wedge \varphi(v))$.

\begin{defi}\label{realizing_set}
Let $A$ be a PAS.
\begin{itemize}
\item[(i)]	We write $\mathcal{P}^\ast A$ for the \emph{set} of inhabited subobjects of $A$, that is, the set of subobjects $U\subseteq A$ such that $\models \exists x:A\s(U(x))$, or equivalently, $U\to 1$ is regular epi.
\item[(ii)]	(Cf.\@ \cite{wps}, Definition 1.2.4.) Let $t(\vec{x},y)$ be a term. We say that $U\in\mathcal{P}^\ast A$ \emph{realizes} $\lambda\vec{x}.t$ if:
\begin{itemize}
\item	$U(r)\models_{r,\vec{a}:A} r\vec{a}\denotes$, and
\item	$t(\vec{a},b)\denotes \wedge\ U(r) \models_{r,\vec{a},b:A} r\vec{a}b\denotes\wedge\ r\vec{a}b = t(\vec{a},b)$.
\end{itemize}
(Here the tuple $\vec{a}$ has the same length as $\vec{x}$.) \ruitje
\end{itemize}
\end{defi}

\begin{defi}
(Cf.\@ \cite{wps}, Definition 1.3.14.) Let $A$ be a PAS.
\begin{itemize}
\item[(i)]	We make $\mathcal{P}^\ast A$ into a PAS over $\mathsf{Set}$ in the following way. For $U,V\in\mathcal{P}^\ast A$, we say that $UV\denotes$ if $U\times V\subseteq D$. In this case, $UV$ is the image of $U\times V\subseteq D\to A$, which is again an element of $\mathcal{P}^\ast A$.
\item[(ii)]	A \emph{filter} $\phi$ on $A$ is a subset of $\mathcal{P}^\ast A$ satisfying the following two conditions:
\begin{itemize}
\item	$\phi$ is upwards closed, i.e., if $U\subseteq V$ and $U\in\phi$, then $V\in \phi$;
\item	$\phi$ is closed under application, i.e., if $U,V\in\phi$ and $UV\denotes$, then $UV\in\phi$.
\end{itemize}
\item[(iii)]	A set $G\subseteq\mathcal{P}^\ast A$ is called \emph{combinatorially complete} if, for every term $t(\vec{x})$, there exists a $U\in G$ realizing $\lambda\vec{x}.t$. 
\item[(iv)]		If $\phi$ is a combinatorially complete filter on $A$, then $(A,\phi)$ is called a \emph{\textup{(}relative\textup{)} partial combinatory algebra} over $\mathcal{C}$ (PCA). \ruitje
\end{itemize}
\end{defi}
\begin{defi}
The set $\mathcal{P}^\ast A$ is always a filter, and it is clearly the largest possible filter. We call this filter the \emph{maximal filter} on $A$. If this filter is combinatorially complete, we say that $A$ is an \emph{absolute} PCA. \ruitje
\end{defi}
\begin{defi}\label{generated_singletons}
(Cf.\@ \cite{wps}, Definition 2.4.1.) More generally, we may select a privileged set $C$ consisting of global sections of $A$ that is closed under application: if $a,b\in C$ and $ab\denotes$, then $ab\in C$. Then \[\phi_C = \{U\subseteq A\mid \exists a\in C\ (a\subseteq U)\}\] is a filter, called the filter \emph{generated} by $C$. We say that such a filter is \emph{generated by singletons}. This filter is combinatorially complete if and only if for every term $t(\vec{x})$, there is an element from $C$ that realizes $\lambda\vec{x}.t$. We also write $(A,C)$ instead of $(A,\phi_C)$.
\ruitje
\end{defi}
\begin{ex}
Suppose that $\mathcal{C}=\mathsf{Set}$, so that a PAS is just a set $A$ equipped with a binary partial function. Then the maximal filter is combinatorially complete if and only if $A$ is a classical PCA. If $C\subseteq A$ is a set of elements of $A$ that is closed under application, then the filter generated by $C$ is combinatorially complete if and only if $(A,C)$ is a classical relative PCA. \ruitje
\end{ex}

In order to work with realizing sets efficiently, we generalize the notation of \cref{realizing_set}.
\begin{defi}\label{convention_parameters}
Let $A$ be a PAS, let $t(y_0,\ldots,y_{m-1},x_0,\ldots,x_n)$ be a term, and let $\vec{U} = U_0,\ldots, U_{m-1}$ be from $\mathcal{P}^\ast A$. If $V\in\mathcal{P}^\ast A$, then we say that $V$ realizes $\lambda\vec{x}.t(\vec{U},\vec{x})$ if there exists a $W$ realizing $\lambda\vec{y}\s\vec{x}.t$ such that $V\subseteq W\vec{U}$. \ruitje
\end{defi}
\begin{rem}
This notation will occasionally create a slight ambiguity. For example, if we say that $V$ realizes $\lambda x.UU$, then this might arise from either $t(y,x) = yy$ or from $t(y,z,x) = yz$. Therefore, we also adopt the following convention: if we write a subobject $U$ of $A$ more than once in a term, then we assume we have substituted all these occurrences for the same variable, that is, we go with the first option. This is only for the sake of definiteness; in practice it does not matter which option one uses, except for the fact that the first option introduces fewer variables.
\ruitje
\end{rem}
With \cref{convention_parameters}, we have the following generalization of combinatorial completeness: if $t(\vec{y},\vec{x})$ is a term, where the tuple $\vec{x}$ is nonempty, and $\vec{U}\in\phi$, then there exists a $V\in\phi$ realizing $\lambda\vec{x}.t(\vec{U},\vec{x})$. Indeed, we may first take a $W\in\phi$ realizing $\lambda\vec{y}\vec{x}.t$ itself, and then set $V=W\vec{U}$. We also note the following important property of realizing sets: if $V$ realizes $\lambda\vec{x}.t(\vec{U},\vec{x})$ and $\vec{W}$ is a tuple of inhabited subobjects from $A$ such that $t(\vec{U},\vec{W})\denotes$, then $V\vec{W}$ is defined as well, and it is a subobject of $t(\vec{U},\vec{W})$.

In the usual way, we have the following result.
\begin{lem}
Let $A$ be a PAS and let $\phi$ be a filter on $A$. Then $(A,\phi)$ is a PCA if and only if $\phi$ contains sets $\mathsf{K}$ and $\mathsf{S}$ realizing $\lambda xy.x$ and $\lambda xyz. xz(yz)$ respectively.\qed
\end{lem}
We will also need a few other common combinators that may be found in $\phi$: an identity combinator $\mathsf{I}$ realizing $\lambda x.x$, a combinator $\overline{\mathsf{K}}$ realizing $\lambda xy.y$, and pairing and unpairing combinators $\mathsf{P},\mathsf{P}_0$ and $\mathsf{P}_1$ realizing $\lambda xyz. zxy$, $\lambda x.x\mathsf{K}$ and $\lambda x. x\overline{\mathsf{K}}$ respectively. For any choice of these pairing and unpairing combinators, we have:
\begin{align*}
\mathsf{P}(p)&\models_{p,a,b:A} pab\denotes,\\
\mathsf{P}(p)\wedge\mathsf{P}_0(p_0)&\models_{p,p_0,a,b:A} p_0(pab) = a,\mbox{ and}\\
\mathsf{P}(p)\wedge\mathsf{P}_1(p_1)&\models_{p,p_1,a,b:A} p_1(pab) = b
\end{align*}
We will assume that, whenever we work with a PCA, we have made some choice of the combinators above in the filter $\phi$.

In \cref{generated_singletons}, we used the idea of generating a filter by a certain subset of $\mathcal{P}^\ast A$. We can generalize this as follows.

\begin{prop}
\textup{(}Cf.\@ \cite{wps}, Example 1.3.20.\textup{)} If $A$ is a PAS and $G\subseteq\mathcal{P}^\ast A$, then there exists a least filter $\langle G\rangle$ extending $G$, given by
\begin{align}\label{generated_filter}
\langle G\rangle = \{V\in\mathcal{P}^\ast A\mid \exists\textup{ term }t(\vec{x})\s \exists \vec{U}\in G\s (t(\vec{U})\denotes\wedge\ t(\vec{U})\subseteq V)\}.
\end{align}
\end{prop}
\begin{proof}
The existence of $\langle G\rangle$ is obvious from the definition of a filter and the fact that $\mathcal{P}^\ast A$ itself is always a filter. Therefore, it remains to prove \eqref{generated_filter}. First of all, we show that the right-hand side of \eqref{generated_filter} is indeed a filter containing $G$. Upwards closure is obvious, so suppose we have $V,V'\in\mathcal{P}^\ast A$, terms $t(\vec{x}), t'(\vec{x}')$, and $\vec{U},\vec{U}'\in G$ such that $t(\vec{U})\denotes$, $t'(\vec{U}')\denotes$, $t(\vec{U})\subseteq V$, $t'(\vec{U}')\subseteq V'$ and $VV'\denotes$. Define the term $s(\vec{x},\vec{x}')$ as $t(\vec{x})\cdot t'(\vec{x}')$. Since $VV'\denotes$, we see that $s(\vec{U},\vec{U}') = t(\vec{U})\cdot t'(\vec{U}')$ denotes as well, and is a subobject of $VV'$, as desired. Finally, it is clear that the right-hand side of \eqref{generated_filter} contains $G$.

Now suppose that $\phi$ is any filter extending $G$. Since $\phi$ is closed under application, it must contain $t(\vec{U})$ whenever $\vec{U}\in G$ and $t(\vec{U})\denotes$. Since $\phi$ is also upwards closed, it must contain the entire right-hand side of \eqref{generated_filter}, which completes the proof.
\end{proof}

Clearly, if $G\subseteq\mathcal{P}^\ast A$ is combinatorially complete, then $(A,\langle G\rangle)$ is a PCA.

\begin{ex}
If $C$ is a set of global sections of $A$ that is closed under application (see \cref{generated_singletons}), then $\langle C\rangle = \phi_C$. \ruitje
\end{ex}

As in the case of $\mathsf{Set}$, we have a category of assemblies, which we define now.
\begin{defi}
Let $(A,\phi)$ be a PCA.
\begin{itemize}
\item[(i)]	An \emph{assembly} $X$ over $(A,\phi)$ is a pair $(|X|,E_X)$, where $|X|$ is an object of $\mathcal{C}$ and $E_X$ is a total relation between $|X|$ and $A$. More explicitly, $E_X$ is a subobject of $|X|\times A$ such that $\models_{x:|X|} \exists a: A\s (E_X(x,a))$, or equivalently, the projection $E_X\to |X|$ is regular epi.
\item[(ii)]	Let $X$ and $Y$ be assemblies. A \emph{morphism of assemblies} $X\to Y$ is an arrow $f\colon |X|\to|Y|$ for which there exists a $U\in\phi$ such that
\[
E_X(x,a)\wedge U(r) \models_{x:|X|; r,a:A} ra\denotes\wedge\ E_Y(f(x),ra)
\] holds. We say that such a $U$ \emph{tracks} $f\colon X\to Y$ \ruitje
\end{itemize}
\end{defi}

\begin{prop}\label{GammaNabla}
Assemblies over a PCA $(A,\phi)$ and morphisms between them form a category $\asm(A,\phi)$, and there exists a pair of functors
\[\begin{tikzcd}
\mathcal{C} \arrow[r, shift right, "\nabla" below] & \asm(A,\phi) \arrow[l, shift right, "\Gamma" above]
\end{tikzcd}\]
with $\Gamma\dashv\nabla$.
\end{prop}
\begin{proof}
If $X$ is an assembly, then $\mathsf{I}$ tracks $\id_{|X|}$ as a morphism $X\to X$. Now let $X\stackrel{f}{\longrightarrow} Y\stackrel{g}{\longrightarrow}Z$ be morphisms, tracked by $U$ and $V$ from $\phi$ respectively, and pick a $W\in\phi$ realizing $\lambda x.V(Ux)$. We claim that $W$ tracks $gf\colon X\to Z$. To this end, select a $W'$ realizing $\lambda yzx.z(yx)$ such that $W \subseteq W'UV$. We reason internally in $\mathcal{C}$: suppose we have $x\in|X|$ and $r,a\in A$ such that $E_X(r,a)$ and $W(r)$. Then there exist $r',s,t\in A$ such that $W'(r')$, $U(s)$, $V(t)$ and $r=r'st$. From $E_X(x,a)$ and $U(s)$, we can conclude that $sa\denotes$ and $E_Y(f(x),sa)$. From the latter and $V(t)$, we can conclude that $t(sa)\denotes$ and $E_Z(g(f(x)),t(sa))$. Since $W'(r')$, we have that $ra = r'sta$ is defined and equal to $t(sa)$, so $E_Z(g(f(x)),ra)$, as desired.

If $Y$ is an object of $\mathcal{C}$, then we define $\nabla Y$ by $|\nabla Y| = Y$ and $E_{\nabla Y} = Y\times A$. This is always an assembly because $A$ is inhabited. Moreover, an arrow $f\colon Y\to Y'$ of $\mathcal{C}$ is always a morphism $\nabla Y\to\nabla Y'$, since it is tracked by $\mathsf{I}$, so this extends to a functor $\nabla\colon\mathcal{C}\to\asm(A,\phi)$.

Similarly, if $X$ is an assembly and $Y$ is an object of $\mathcal{C}$, then any arrow $f\colon |X|\to Y$ is automatically a morphism $X\to\nabla Y$, since it is always tracked by $\mathsf{I}$. This shows that $\Gamma\dashv \nabla$.
\end{proof}

\begin{rem}
The proof above that $gf$ is a morphism proceeded as follows: first, we constructed the desired tracker $W\in\phi$ by mimicking the usual construction of this tracker in the case of classical PCAs. Then we unpacked all the definitions, and finally, we gave an internal argument that this is indeed a tracker, and this argument parallels the usual argument for classical PCAs. In the sequel, we will usually only show how to construct a desired element of $\phi$, leaving the unpacking of the definitions and the internal reasoning needed to verify that this element meets the appropriate conditions to the reader. \ruitje
\end{rem}

\begin{ex}
If $(A,C)$ is a classical relative PCA, then $\asm(A,C)$ as defined above coincides with the familiar category of assemblies for $(A,C)$. \ruitje
\end{ex}

\cref{GammaNabla} implies in particular that $\nabla 1$ is the terminal object of $\asm(A,\phi)$. The notation $\Gamma$ comes from the fact that, as we mentioned in the Introduction, in the case of a classical absolute PCA over $\mathsf{Set}$, this $\Gamma$ is the global sections functor. For classical relative PCAs, this is no longer true. In fact, we can use global sections to determine whether a given PCA is absolute.
\begin{prop}\label{absoluteness}
A PCA $(A,\phi)$ is absolute if and only if $\Gamma$ commutes \textup{(}up to isomorphism\textup{)} with the global sections functors $\asm(A,\phi)\to\mathsf{Set}$ and $\mathcal{C}\to\mathsf{Set}$.
\end{prop}
\begin{proof}
Since $\Gamma$ is faithful, it commutes with the global sections functors if and only if, for every assembly $X$, any global section of $|X|$ is also a global section of $X$. First, suppose that $(A,\phi)$ is absolute and that $x\colon 1\to |X|$ is a global section. Then $U:=\{a\in A\mid E_X(x,a)\}$ is inhabited, so it belongs to $\phi$. Since $\mathsf{K}U$ tracks $x\colon 1\to X$, it follows that $x$ is also a global section of $X$.

Conversely, suppose that $(A,\phi)$ is \emph{not} absolute. Then there exists an inhabited subobject $U$ of $A$ that does not belong to $\phi$. Define the assembly $X$ by $|X|=1$ and $E_X = U \subseteq A \simeq 1\times A$; this is indeed an assembly since $U$ is inhabited. Moreover, $|X|=1$ has the global section $!\colon 1\to 1$. If $!$ were also a morphism $1\to X$ tracked by $V$, then we would have $V\!A\denotes$ and $V\!A \subseteq U$. Since $\phi$ is nonempty (it must contain $\mathsf{K}$) and upwards closed, it always contains $A$ itself. This means that $V\!A\in\phi$, but $U\not\in\phi$, which contradicts the upwards closure of $\phi$.
\end{proof}

\section{Applicative Morphisms and Transformations}\label{section_appmor}

In \cite{wps}, Stekelenburg generalizes Longley's definition of applicative morphisms between classical PCAs (\cite{longleyphd}, Definition 2.1.1) to applicative morphisms between PCAs constructed over a certain (Heyting) category $\mathcal{C}$. The goal of this section will be to generalize this to PCAs constructed over possibly different regular categories. As a preliminary to this, we again fix a regular category $\mathcal{C}$, and we define the category of PCAs over $\mathcal{C}$.

As usual, a relation between two objects $A$ and $B$ of $\mathcal{C}$ is a subobject $f\subseteq A\times B$. These may be composed: if $f\subseteq A\times B$ and $g\subseteq B\times C$, then we define their composition $gf\subseteq A\times C$ as
\[\{(a,c)\in A\times C\mid \exists b: B\s (f(a,b)\wedge g(b,c))\}.\]
This composition is associative and for each object $A$, we have the diagonal relation $\delta_A\subseteq A\times A$, which acts as an identity.

\begin{defi}\label{applicative_morphism}
(Cf. \cite{wps}, Definition 2.3.20.) Let $(B,\psi)$ be a PCA and let $A$ be an object of $\mathcal{C}$.
\begin{itemize}
\item[(i)]	If $f,f'\subseteq A\times B$, then we say that $f\leq f'$ if there exists a $U\in\phi$ such that
\[f(a,b)\wedge U(r)\models_{a:A; r,b:B} rb\denotes\wedge\ f'(a,rb).\]
We say that such a $U$ realizes the inequality $f\leq f'$.
\item[(ii)]	Suppose that $A$ is equipped with a partial applicative structure. A relation $f\subseteq A\times B$ is an \emph{applicative premorphism} $A\to (B,\psi)$ if:
\begin{itemize}
\item[(a)]	$f$ is total, i.e., $\models_{a:A} \exists b: B\s (f(a,b))$;
\item[(b)]	there exists a $U\in\psi$ such that
\[f(a,b)\wedge f(a',b)\wedge aa'\denotes \wedge\ U(r) \models_{a,a':A; r,b,b':B} rbb'\denotes\wedge f(aa',rbb').\]
\end{itemize}
A set $U$ as in item (b) is said to \emph{track} the applicative premorphism $f\colon A\to (B,\psi)$.
\item[(iii)]	Suppose that $(A,\phi)$ is a PCA. We say that a relation $f\subseteq A\times B$ is an applicative premorphism $(A,\phi)\to(B,\psi)$ if it is one $A\to (B,\psi)$. The relation $f$ is called an \emph{applicative morphism} if it is an applicative premorphism, and moreover:
\begin{itemize}
\item[(c)]	if $U\in\phi$, then $f(U):=\{b\in B\mid\exists a:A\s (U(a)\wedge f(a,b))\}\in\psi$. \ruitje
\end{itemize}
\end{itemize}
\end{defi}

Applicative premorphisms between PCAs are not particularly well-behaved; for example, they do not form a category. If we restrict to applicative morphisms, on the other hand, we do get a well-behaved category (see also \cite{wps}, p.\@ 69). We prove this as part of the next proposition, which explains the compatibility between the notions from \cref{applicative_morphism} and composition.

\begin{prop}\label{compose}
Let $(B,\psi)$ be a PCA, let $A$ be an object of $\mathcal{C}$.
\begin{itemize}
\item[\textup{(}i\textup{)}]	$\leq$ is a preorder on the set of relations between $A$ and $B$.
\end{itemize}
Now let $f,f'\subseteq A\times B$ such that $f\leq f'$.
\begin{itemize}
\item[\textup{(}ii\textup{)}]	If $g\colon (B,\psi)\to (C,\chi)$ is an applicative morphism, then $gf\leq gf'$.
\item[\textup{(}iii\textup{)}]	If $g\subseteq C\times A$ is any relation, then $fg\leq f'g$.
\end{itemize}
Now suppose that $A$ is a PAS and that $f$ is an applicative premorphism $A\to(B,\psi)$.
\begin{itemize}
\item[\textup{(}iv\textup{)}]	If $g\colon (B,\psi)\to (C,\chi)$ is an applicative morphism, then $gf$ is an applicative premorphism $A\to (C,\chi)$.
\end{itemize}
Finally, suppose that $(A,\phi)$ is a PCA and that $f\colon (A,\phi)\to(B,\psi)$ is an applicative morphism.
\begin{itemize}
\item[\textup{(}v\textup{)}]	If $g\colon (B,\psi)\to (C,\chi)$ is an applicative morphism, then $gf$ is an applicative morphism $(A,\phi)\to (C,\chi)$.
\end{itemize}
In particular, PCAs over $\mathcal{C}$ and applicative morphisms form a preorder-enriched category, which we denote by $\mathsf{PCA}_\mathcal{C}$.
\end{prop}
\begin{proof}
If $f\subseteq A\times B$, then $\mathsf{I}$ always realizes $f\leq f$. Now suppose we have $f,f',f''\subseteq A\times B$ such that $f\leq f'\leq f''$, and say that $U,V\in \psi$ realize $f\leq f'$ and $f'\leq f''$ respectively. Then any realizer of $\lambda x. V(Ux)$ also realizes $f\leq f''$, as desired.

For statement (ii), let $U\in\psi$ realize $f\leq f'$ and let $V\in\chi$ track $g$. Then $g(U)\in\chi$, and any realizer of $\lambda x.V(g(U)\cdot x)$ also realizes $gf\leq gf'$.

For statement (iii), we simply observe that every realizer of $f\leq f'$ also realizes $fg\leq f'g$.

For statement (iv), it is easy to check that $gf$ satisfies requirement (a). For requirement (b), let $U\in\psi$ track $f$ and let $V\in \chi$ track $g$. Then any realizer of $\lambda xy.V(V \cdot g(U)\cdot x)y$ tracks $gf$.

For statement (v), we observe that $(gf)(U) = g(f(U))$ for every $U\subseteq A$, so requirement (c) follows.

For the final claim, it suffices to show that, if $(A,\phi)$ is a PCA, then the identity relation $\delta_A$ is always an applicative morphism. Requirement (a) is clear. For requirement (b), we observe that $\mathsf{I}$ always tracks $\delta_A$, and for requirement (c), we use that $\delta_A(U)=U$ for every $U\subseteq A$. This completes the proof.
\end{proof}

The following proposition simplifies the definition of an applicative morphism if we work with generated filters.
\begin{prop}\label{lift_left}
Suppose that $A$ is a PAS, that $G\subseteq\mathcal{P}^\ast A$ is combinatorially complete, that $(B,\psi)$ is a PCA and that $f\colon A\to (B,\psi)$ is an applicative premorphism. Then $f$ is an applicative morphism $(A,\langle G\rangle)\to(B,\psi)$ if and only if it satisfies:
\begin{itemize}
\item[\textup{(}c'\textup{)}]	if $U\in G$, then $f(U)\in\psi$.
\end{itemize}
\end{prop}
\begin{proof}
Since $G\subseteq \langle G\rangle$, any applicative morphism must satisfy (c').

For the converse, suppose that (c') holds. Define
\[f^{-1}(\psi) := \{U\subseteq A\mid f(U)\in\psi\}.\]
We claim that $f^{-1}(\psi)$ is a filter on $A$. First of all, since every element of $\psi$ is inhabited, every element of $f^{-1}(\psi)$ must be inhabited as well. If $U\in f^{-1}(\psi)$ and $U\subseteq V$, then we see that $f(U)\subseteq f(V)$. Since $f(U)\in \psi$, it follows that $f(V)\in \psi$ as well, so $V\in f^{-1}(\psi)$. Finally, suppose that $U,U'\in f^{-1}(\psi)$ and $UU'\denotes$. If $W\in\psi$ tracks $f$, then $W\cdot f(U)\cdot f(U')$ is defined and a subobject of $f(UU')$. But we know that $W, f(U), f(U')\in\psi$, so since $\psi$ is a filter, we can conclude that $f(UU')\in\psi$ as well, as desired.

Requirement (c') means that $G\subseteq f^{-1}(\psi)$. By what we have just shown, it follows that $\langle G\rangle \subseteq f^{-1}(\psi)$, which means that $f\colon(A,\langle G\rangle)\to(B,\psi)$ satisfies requirement (c).
\end{proof}

We investigate some of the structure of $\mathsf{PCA}_\mathcal{C}$; these generalize known properties of the category $\mathsf{PCA}$.

\begin{prop}
For a regular category $\mathcal{C}$, the category $\mathsf{PCA}_\mathcal{C}$ has a pseudozero object.
\end{prop}
\begin{proof}
First of all, consider the terminal object $1\in\mathcal{C}$, equipped with the total application map $1\times 1\to 1$ and the maximal filter $\{1\}$. This is a PCA, since $1$ realizes $\lambda\vec{x}.t$ for every term $t(\vec{x})$. Suppose that $(A,\phi)$ is an object of $\mathsf{PCA}_\mathcal{C}$. If $f\subseteq A\simeq A\times 1$, then $f$ always satisfies requirement (b) for an applicative morphism $(A,\phi)\to 1$, and it satisfies requirement (a) if and only if $f=A$. In this case, (c) is also satisfied, since every element of $\phi$ is inhabited. This shows that 1 is in fact a 1-terminal object, which automatically makes it a pseudoterminal object.

On the other hand, $f\subseteq A\simeq 1\times A$ always satisfies requirement (b) for an applicative morphism $1\to (A,\phi)$, and it satisfies requirement (c) if and only if $f\in\phi$. In this case, requirement (a) is also satisfied, again since every element of $\phi$ is inhabited. Moreover, it is easily seen that all these possible $f$ are isomorphic, so 1 is a pseudoinitial object in $\mathsf{PCA}_\mathcal{C}$.
\end{proof}

This means that $\mathsf{PCA}_\mathcal{C}$ also has zero morphisms. In fact, the zero morphism $(A,\phi)\to (B,\psi)$ is the top element of $\mathsf{PCA}_\mathcal{C}((A,\phi),(B,\psi))$, which we can represent by $A\times B$ itself. Moreover, $f\colon (A,\phi)\to (B,\psi)$ is a zero morphism if and only if $A\times U\subseteq f$ for some $U\in\psi$. Intuitively, we can view this as the applicative morphism that carries no information: every element of $A$ is represented by the same realizing subset of $B$.

\begin{prop}
The category $\mathsf{PCA}_\mathcal{C}$ has finite pseudocoproducts.
\end{prop}
\begin{proof}
We have already seen that $\mathsf{PCA}_\mathcal{C}$ has a pseudo-initial object, so consider two PCAs $(A,\phi)$ and $(B,\psi)$. We equip their product $A\times B$ in $\mathcal{C}$ with the coordinatewise application map, that is, we have $D_{A\times B}(a,b,a',b')$ if and only if $D_A(a,a')\wedge D_B(b,b')$, and in this case, $(a,b)\cdot (a',b') = (aa',bb')$. Define $\phi\times\psi = \{U\times V\subseteq A\times B\mid U\in\phi, V\in\psi\}\subseteq\mathcal{P}^\ast(A\times B)$. If $t(\vec{x})$ is a term and $U\in\phi$ and $V\in\psi$ realize $\lambda\vec{x}.t$ with respect to $A$ and $B$ respectively, then $U\times V$ realizes $\lambda\vec{x}.t$ with respect to $A\times B$. This shows that $\phi\times\psi$ is combinatorially complete, so $(A\times B, \langle\phi\times \psi\rangle)$ is a PCA. We claim that this is the pseudocoproduct of $(A,\phi)$ and $(B,\psi)$.

First of all, we have an applicative morphism $\kappa_0\colon (A,\phi)\to (A\times B, \langle\phi\times \psi\rangle)$, defined by $\kappa_0 = \{(a,a',b)\in A\times (A\times B)\mid a=a'\}$. Requirements (a) and (c) are obviously satisfied, and we observe that $\mathsf{I}\times \mathsf{K}$ tracks $\kappa_0$, so requirement (b) is satisfied as well. We define $\kappa_1\colon (B,\psi)\to (A\times B, \langle\phi\times \psi\rangle)$ analogously.

Now suppose that applicative morphisms $f\colon (A,\phi)\to(C,\chi)$ and $g\colon (B,\psi)\to(C,\chi)$ are given. We define $[f,g]\colon (A\times B, \langle\phi\times \psi\rangle)\to (C,\chi)$ as
\[\{(a,b,c)\in (A\times B)\times C\mid \exists p,c',c'': C\s (\mathsf{P}(p)\wedge f(a,c')\wedge g(b,c'')\wedge pc'c''=c)\}.\]

Requirement (a) is clearly satisfied. For requirement (b), if $U,V\in\chi$ track $f$ and $g$ respectively, then any realizer of
\[\lambda xy.\mathsf{P}(U(\mathsf{P}_0 x)(\mathsf{P}_0 y))(V(\mathsf{P}_1 x)(\mathsf{P}_1 y)).\]
tracks $[f,g]$. If $U\in \phi$ and $V\in\psi$ are arbitrary, then $[f,g](U\times V) = \mathsf{P} f(U) g(V)\in \chi$, so by \cref{lift_left}, requirement (c) is satisfied as well.

Furthermore, we have that $\mathsf{P}_0$ realizes $[f,g]\circ \kappa_0\leq f$, whereas any realizer of $\lambda x.\mathsf{P}\cdot x\cdot g(B)$ realizes $f\leq [f,g]\circ \kappa_0$. This means that $[f,g]\circ\kappa_0\simeq f$, and similarly, we prove that $[f,g]\circ\kappa_1\simeq g$.

Now suppose that we have applicative morphisms $h,k\colon (A\times B, \langle\phi\times \psi\rangle)\to (C,\chi)$ such that $h\kappa_0\leq k\kappa_0$ and $h\kappa_1\leq k\kappa_1$. Let $U\in\chi$ be a tracker of $k$, and let $V,W\in\chi$ realize $h\kappa_0\leq k\kappa_0$ and $h\kappa_1\leq k\kappa_1$ respectively. Then any realizer of
\[\lambda x.U(U\cdot k(\mathsf{K}\times\overline{\mathsf{K}})\cdot(Vx))(Wx)\]
realizes the inequality $h\leq k$, which finishes the proof.
\end{proof}

\begin{rem}\label{theproductdoesnotexist!}
One may wonder whether $(A\times B, \langle\phi\times \psi\rangle)$ is also the product of $(A,\phi)$ and $(B,\psi)$. It turns out that this is only true in a weak sense. We can define a projection map $\pi_0\colon (A\times B,\langle \phi\times \psi\rangle)\to (A,\phi)$ by $\pi_0 = \{(a,b,a')\in (A\times B)\times A\mid a=a'\}$; define $\pi_1$ similarly. If $f\colon (C,\chi)\to (A,\phi)$ and $g\colon (C,\chi)\to (B,\psi)$ are applicative morphisms, then we get a new morphism $\langle f,g\rangle \colon (C,\chi)\to (A\times B,\langle \phi\times \psi\rangle)$, defined by
\[\langle f,g\rangle = \{(c,a,b)\in C\times (A\times B)\mid f(c,a)\wedge g(c,b)\}.\]
This morphism satisfies $\pi_0\circ\langle f,g\rangle \simeq f$ and $\pi_1\circ\langle f,g\rangle \simeq g$ (we even have equality here), but it is not necessarily essentially unique with this property. If $h\colon(C,\chi)\to (A\times B,\langle \phi\times \psi\rangle)$ satisfies $\pi_0h\simeq f$ and $\pi_1h\simeq g$ (or even with equality), then we can only guarantee that $h\leq \langle f,g\rangle$, but not that $\langle f,g\rangle\leq h$. 

Later, we shall pass to a larger category $\mathsf{PCA}$, in which we \emph{can} form finite products, and even all small products).
\ruitje
\end{rem}

We now proceed to consider PCAs constructed over different regular categories. In the following, $\mathcal{C}, \mathcal{D}$ and $\mathcal{E}$ will always denote regular categories.

In order to move between two regular categories $\mathcal{C}$ and $\mathcal{D}$, we consider regular functors $p\colon \mathcal{C}\to\mathcal{D}$. These functors preserve the interpretation of regular formulas, and as a result, they preserve the validity of regular sequents.

If $A$ is a PAS over $\mathcal{C}$ and $G\subseteq\mathcal{P}^\ast A$, then we write $p(G) = \{p(U)\mid U\in G\}$, which is a subset of $\mathcal{P}^\ast(p(A))$, since $p$ preserves inhabited objects.

\begin{prop}\label{HF_on_objects}
Let $p\colon \mathcal{C}\to\mathcal{D}$ be a regular functor, let $A$ be a PAS over $\mathcal{C}$, and let $G\subseteq\mathcal{P}^\ast A$. Then:
\begin{itemize}
\item[\textup{(}i\textup{)}]	$p(A)$ is a PAS over $\mathcal{D}$;
\item[\textup{(}ii\textup{)}]	$\langle p(\langle G\rangle)\rangle = \langle p(G)\rangle$;
\item[\textup{(}iii\textup{)}]	if $G$ is combinatorially complete, then so is $p(G)$;
\item[\textup{(}iv\textup{)}]	if $(A,\phi)$ is a PCA over $\mathcal{C}$, then $(p(A),\langle p(\phi)\rangle)$ is a PCA over $\mathcal{D}$.
\end{itemize}
\end{prop}
\begin{proof}
(i) Since $p$ is left exact, the object $p(A)$ inherits a partial applicative structure from $A$ in the obvious way. Explicitly, its domain is $p(D)\subseteq p(A)\times p(A)$, and the required map $p(D)\to p(A)$ is the image of the application map $D\to A$ under $p$.

(ii) First of all, we observe that $p(G) \subseteq p(\langle G\rangle) \subseteq \langle p(\langle G\rangle)\rangle$, which implies $\langle p(G)\rangle \subseteq \langle p(\langle G\rangle)\rangle$. For the converse, suppose that we have an element $V\in p(\langle G\rangle)$. This means that there exist a $V'\subseteq A$, a term $t(x_0, \ldots, x_n)$ and $U_0, \ldots, U_n\in G$ such that $t(\vec{U})\!\denotes$, $t(\vec{U})\subseteq V'$ and $p(V') = V$. Now we observe that $t(p(U_0), \ldots, p(U_n))$ also denotes, and
\[t(p(U_0), \ldots, p(U_n)) = p(t(U_0, \ldots, U_n)) \subseteq p(V') = V.\]
Since $p(U_i)\in p(G)$ for each $i$, this yields that $V\in \langle p(G)\rangle$. We conclude that $p(\langle G\rangle) \subseteq \langle p(G)\rangle$, hence also $\langle p(\langle G\rangle)\rangle \subseteq \langle p(G)\rangle$.

(iii)	Let $t(\vec{x})$ be a term, and suppose that $U\in G$ realizes $\lambda\vec{x}.t$ w.r.t.\@ $A$. Since $p$ preserves apoplication and the validity of regular sequents, it follows that $p(U)\in p(G)$ realizes $\lambda\vec{x}.t$ w.r.t.\@ $p(A)$, as desired.

(iv) now follows immediately from (iii).
\end{proof}

\begin{rem}
In the proof above, we used the following fact: if $U,V\in\mathcal{P}^\ast A$ satisfy $UV\!\denotes$, then $p(U)\cdot p(V)$ is defined as well, and equal to $p(UV)$. The converse does \emph{not} hold, i.e., we can have that $p(U)\cdot p(V)\denotes$ without $UV$ being defined. To see this, consider for example the unique functor $p\colon\mathcal{C}\to\mathbf{1}$. \ruitje
\end{rem}

For a PCA $(A,\phi)$ over $\mathcal{C}$ and a regular functor $p\colon \mathcal{C}\to\mathcal{D}$, we shall denote the PCA $(p(A),\langle p(\phi)\rangle)$ by $p^\ast(A,\phi)$. Our next goal is to define $p^\ast$ on applicative morphisms.
\begin{prop}\label{HF_on_arrows}
Let $p\colon \mathcal{C}\to\mathcal{D}$ be a regular functor, let $A$ be an object of $\mathcal{C}$, let $(B,\psi)$ be a PCA over $\mathcal{C}$, and let $f,f'\subseteq A\times B$.
\begin{itemize}
\item[\textup{(}i\textup{)}]	If $f\leq f'$, then $p(f)\leq p(f')$. \textup{(}Here we see $p(f)$ and $p(f')$ as relations between $p(A)$ and $p(B)$, the latter being the underlying object of $p^\ast(B,\psi)$.\textup{)}
\item[\textup{(}ii\textup{)}]	If $A$ is a PAS and $f\colon A\to(B,\psi)$ is an applicative premorphism, then $p(f)$ is an applicative premorphism $p(A)\to p^\ast(B,\psi)$.
\item[\textup{(}iii\textup{)}]	If $(A,\phi)$ is a PCA and $f\colon (A,\phi)\to(B,\psi)$ is an applicative morphism, then $p(f)$ is an applicative morphism $p^\ast(A,\phi)\to p^\ast(B,\psi)$.
\end{itemize}
\end{prop}
\begin{proof}
(i) Since $p$ preserves application and regular sequents, we have that $p(U)$ realizes $p(f)\leq p(f')$ whenever $U\in\phi$ realizes $f\leq f'$.

(ii) Requirement (a) follows since $p$ preserves total relations. For requirement (b), we again use the fact that $p$ preserves regular sequents to see that: if $U\in\phi$ tracks $f$, then $p(U)$ tracks $p(f)$.

(iii)	Suppose that $U\subseteq A$. Since $p$ preserves regular logic, we have that $p(f(U)) = p(f)(p(U))$. Therefore, if $U\in\phi$, then $f(U)\in\psi$, so $p(f)(p(U)) = p(f(U))\in p(\psi)\subseteq \langle p(\psi)\rangle$. By \cref{lift_left}, $p(f)$ also satisfies requirement (c).
\end{proof}

If $f\colon (A,\phi)\to (B,\psi)$ is an arrow of $\mathsf{PCA}_\mathcal{C}$, then we write $p^\ast f$ for the applicative morphism $p(f)\colon p^\ast(A,\phi)\to p^\ast(B,\psi)$.
\begin{thm}
If $p\colon \mathcal{C}\to\mathcal{D}$ is a regular functor, then $p^\ast$ is a preorder-enriched functor $\mathsf{PCA}_\mathcal{C}\to\mathsf{PCA}_\mathcal{D}$. Moreover, the construction $p\mapsto p^\ast$ is functorial.
\end{thm}
\begin{proof}
Since $p$ is left exact, we have that $p(\delta_A) = \delta_{p(A)}$ for every object $A$ of $\mathcal{C}$. Moreover, if $f\subseteq A\times B$ and $g\subseteq B\times C$, then $p(gf) = p(g)\circ p(f)$ since $p$ preserves regular logic. Together with \cref{HF_on_arrows}, this implies that $p^\ast$ is a preorder-enriched functor.

It is clear that $\id_\mathcal{C}^\ast = \id_{\mathsf{PCA}_\mathcal{C}}$. Now suppose that we have regular functors $\mathcal{C}\stackrel{p}{\longrightarrow}\mathcal{D}\stackrel{q}{\longrightarrow}\mathcal{E}$ and a PCA $(A,\phi)$ over $\mathcal{C}$. The partial applicative structures $q(p(A))$ and $(qp)(A)$ clearly coincide. Moreover, by \cref{HF_on_objects}(ii), we have
\[\langle q(\langle p(\phi)\rangle)\rangle = \langle q(p(\phi))\rangle = \langle (qp)(\phi)\rangle,\]
so $q^\ast(p^\ast(A,\phi)) = (qp)^\ast(A,\phi)$. Finally, if $f$ is an arrow of $\mathsf{PCA}_\mathcal{C}$, then $q^\ast(p^\ast(f))$ and $(qp)^\ast(f)$ are clearly the same relation, which completes the proof.
\end{proof}
Let $\mathsf{REG}$ denote the category of locally small regular categories and regular functors. We may, of course, consider this as a 2-category, the 2-cells being natural transformations between regular functors. One might wonder whether $(-)^\ast$ can be extended to a 2-functor from regular categories into (preorder-enriched) categories. This does not seem to be the case, but we can get a partial result, which suffices for our purposes.

Suppose that $p,q\colon \mathcal{C}\to\mathcal{D}$ are regular functors and that $\mu\colon p\To q$ is a natural transformation. If $(A,\phi)$ is a PCA over $\mathcal{C}$, then $\mu_A$ is an arrow $p(A)\to q(A)$. We can also view that arrow as its graph, which is the (single-valued) relation
\[\{(a,b)\in p(A)\times q(A)\mid \mu_A(a) = b\},\]
that we denote by $\overline{\mu_A}$.
\begin{prop}\label{transporting_mu}
Let $p,q\colon \mathcal{C}\to\mathcal{D}$ be regular functors and let $\mu\colon p\To q$ be a natural transformation.
\begin{itemize}
\item[\textup{(}i\textup{)}]	If $(A,\phi)$ is a PCA over $\mathcal{C}$, then $\overline{\mu_A}$ is an applicative premorphism $p^\ast(A,\phi)\to q^\ast(A,\phi)$.
\item[\textup{(}ii\textup{)}]	If $f\colon (A,\phi)\to (B,\psi)$ is an applicative morphism, then $\overline{\mu_B}\circ p(f)\leq q(f)\circ\overline{\mu_A}$.
\item[\textup{(}iii\textup{)}]	If $\mu$ is an isomorphism, then $\mu^\ast_{(A,\phi)}:=\overline{\mu_A}$ yields a natural isomorphism $\mu^\ast\colon p^\ast\To q^\ast$.
\end{itemize}
\end{prop}
\begin{proof}
(i)  Requirement (a) clearly holds. By applying the naturality of $\mu$ to the inclusion $D\subseteq A\times A$ and the application map $D\to A$, we can see that
\[ab\denotes\ \models_{a,b:p(A)} \mu_A(a)\cdot\mu_A(b)\denotes\wedge\ \mu_A(a)\cdot\mu_A(b)=\mu_A(ab).\]
This means that $\mathsf{I}$ tracks $\overline{\mu_A}$, so requirement (b) is satisfied as well.

(ii) By applying the naturality of $\mu$ to the inclusion $f\subseteq A\times B$, we see that
\[p(f)(a,b) \models_{a:p(A); b:p(B)} q(f)(\mu_A(a),\mu_ B(b)).\]
This implies that $\overline{\mu_B}\circ p(f)\subseteq q(f)\circ\overline{\mu_A}$, so in particular, we have $\overline{\mu_B}\circ p(f)\leq q(f)\circ\overline{\mu_A}$.

(iii)	First of all, if $U\in\phi$, then the naturality square for the inclusion $U\subseteq A$ tells us that $\overline{\mu_A}(p(U)) = q(U) \in q(\phi)\subseteq \langle q(\phi)\rangle$. By \cref{lift_left}, $\overline{\mu_A}$ is an applicative morphism $p^\ast(A,\phi)\to q^\ast(A,\phi)$. Moreover, this morphism is clearly invertible, its inverse being $\overline{\mu^{-1}_A}$.

If $f\colon (A,\phi)\to (B,\psi)$ is an applicative morphism, then we already know that $\overline{\mu_B}\circ p(f)\subseteq q(f)\circ\overline{\mu_A}$. For the converse inclusion, reason inside $\mathcal{D}$ and let $a\in p(A)$ and $b\in q(B)$ such that $(q(f)\circ\overline{\mu_A})(a,b)$, i.e., $q(f)(\mu_A(a),b)$. Then there exists a $b'\in p(B)$ such that $\mu_B(b') = b$. The naturality diagram for the inclusion $f\subseteq A\times B$ now allows us to conclude that $p(f)(a,b')$. Together with $\mu_B(b') = b$, this yields $(\overline{\mu_B}\circ p(f))(a,b)$, as desired.
\end{proof}
\begin{rem}\label{no_transport}
The relation $\overline{\mu_A}$ does not, in general, seem to be an applicative morphism. If $U\subseteq A$, then the naturality of $\mu$ only guarantees that $\overline{\mu_A}(p(U))\subseteq q(U)$, but not that equality holds. So it could be possible that $\overline{\mu_A}$ sends an element of $p(\phi)$ to a subobject of $q(A)$ that is too small to be in $\langle q(\phi)\rangle$. \ruitje
\end{rem}

At this point, we have enough data to perform a Grothendieck construction to obtain a category $\mathsf{PCA}$ together with an forgetful functor $\mathsf{PCA}\to\mathsf{REG}$ whose fiber above $\mathcal{C}$ is exactly $\mathsf{PCA}_\mathcal{C}$.
\begin{defi}\label{PCA}
The 2-category $\mathsf{PCA}$ is defined as follows.
\begin{itemize}
\item[(i)]	The objects are triples $(\mathcal{C},A,\phi)$, where $(A,\phi)$ is a PCA over the regular category $\mathcal{C}$. We will usually just write $(A,\phi)$ instead of $(\mathcal{C},A,\phi)$.
\item[(ii)]	If $(A,\phi)$ and $(B,\psi)$ are PCAs over $\mathcal{C}$ and $\mathcal{D}$ respectively, then an arrow $(A,\phi)\to(B,\psi)$ is a pair $(p,f)$, where $p\colon \mathcal{C}\to\mathcal{D}$ is a regular functor and $f\colon p^\ast(A,\phi)\to (B,\psi)$ is an applicative morphism;
\item[(iii)]	A 2-cell $(p,f)\To(q,g)$ is a natural transformation $\mu\colon p\To q$ such that $f\leq g\circ \overline{\mu_A}$.
\end{itemize}
An arrow of $\mathsf{PCA}$ is also called an \emph{applicative morphism}, whereas a 2-cell is called an \emph{applicative transformation}. \ruitje
\end{defi}
\begin{conv}
When writing $(A,\phi)$ instead of $(\mathcal{C},A,\phi)$, we drop the reference to the underlying regular category $\mathcal{C}$ of $(A,\phi)$. For example, we write that $(p,f)\colon (A,\phi)\to (B,\psi)$ is an applicative morphism, thereby understanding that $p$ is a regular functor between the underlying regular categories of $(A,\phi)$ and $(B,\psi)$. We will only specify the underlying category when necessary, i.e., when it plays a role in the argument. \ruitje
\end{conv}

As usual, we identify an arrow $f$ of $\mathsf{PCA}_\mathcal{C}$ with the arrow $(\id_\mathcal{C},f)$ of $\mathsf{PCA}$. In particular, every applicative morphism in the sense of \cref{applicative_morphism} is also an applicative morphism in the sense of \cref{PCA}. Concerning item (iii), since the applicative premorphism $\overline{\mu_A}$ is always single-valued, we can easily formulate the given inequality directly. If $(p,f),(q,g)\colon (A,\phi)\to(B,\psi)$, then $U\in\psi$ realizes the inequality $f\leq g\circ\overline{\mu_A}$ if and only if
\[f(a,b)\wedge U(r) \models_{a:p(A); r,b:B} rb\denotes\wedge\ g(\mu_A(a),rb).\]
We will also say that such a $U$ tracks the applicative transformation $\mu$.

\begin{thm}
$\mathsf{PCA}$ as defined above is indeed a 2-category.
\end{thm}
\begin{proof}
The composition of two applicative morphisms $(A,\phi)\stackrel{(p,f)}{\longrightarrow}(B,\psi)\stackrel{(q,g)}{\longrightarrow}(C,\chi)$ is given by $(qp, g\circ q^\ast f) = (qp, g\circ q(f))$. It is well-known from the theory of Grothendieck fibrations that this yields a 1-category.

We define the vertical and horizontal composition of 2-cells as in $\mathsf{REG}$. Suppose that we have parallel applicative morphisms $(p,f),(q,g),(r,h)\colon (A,\phi)\to(B,\psi)$ and that \[(p,f)\stackrel{\mu}{\To} (q,g)\stackrel{\nu}{\To} (r,h)\] are applicative transformations. Then $\nu\mu$ is an applicative transformation $(p,f)\To(r,h)$ since $f\leq g\circ \overline{\mu_A}\leq h\circ\overline{\nu_A}\circ\overline{\mu_A} = h\circ\overline{\nu_A\mu_A}$. 

Now suppose that $(p,f),(q,g)\colon (A,\phi)\to(B,\psi)$ are applicative morphisms and that $\mu$ is an applicative transformation $(p,f)\To (q,g)$. Let $(r,h)\colon (B,\psi)\to (C,\chi)$ be another applicative morphism. Since $r$ is left exact, we have $\overline{r(\mu_A)} = r(\overline{\mu_A})$. Now we see that
\[h\circ r(f) \leq h\circ r(g\circ\overline{\mu_A}) = h\circ r(g) \circ r(\overline{\mu_A}) =h\circ r(g) \circ  \overline{r(\mu_A)} ,\]
so $r\mu$ is an applicative transformation $(r,h)\circ (p,f) \To (r,h)\circ(q,g)$.

On the other hand, if $(r,h)\colon (C,\chi)\to (A,\phi)$ is another applicative morphism, then
\[f\circ p(h) \leq g\circ\overline{\mu_A}\circ p(h) \leq g\circ q(h)\circ \overline{\mu_{r(C)}},\]
so $\mu r$ is an applicative transformation $(p,f)\circ(r,h)\To (q,g)\circ (r,h)$.

All the required equations for 2-cells in $\mathsf{PCA}$ are inherited from $\mathsf{REG}$.
\end{proof}



As we saw in \cref{theproductdoesnotexist!}, the fibers $\mathsf{PCA}_\mathcal{C}$ have products only in a weak sense. By passing to $\mathsf{PCA}$, we can form all small products.
\begin{prop}
The category $\mathsf{PCA}$ has small 2-products.
\end{prop}
\begin{proof}
Suppose that we have a collection $(A_i,\phi_i)$ indexed by a set $I$, where $(A_i,\phi_i)$ is a PCA over $\mathcal{C}_i$. Consider the product category $\mathcal{C}=\prod_{i\in I}\mathcal{C}_i$, which is also a locally small regular category in which regular formulas are interpreted coordinatewise. The object $A = (A_i)_{i\in I}$ is a PAS over $\mathcal{C}$ in the obvious way, and we can define a combinatorially complete filter $\phi$ on $A$ by
\[\phi = \{U=(U_i)_{i\in I}\mid \forall i\in I\s (U_i\in\phi_i)\}.\]
For each $i\in I$, we have the projection map $p_i\colon \mathcal{C}\to\mathcal{C}_i$, which satisfies $p_i(A) = A_i$ and $p_i(U) = U_i$ for every subobject $U$ of $A$. This implies that $(p_i,\delta_{A_i})$ is an applicative morphism $(A,\phi)\to (A_i,\phi_i)$.

First, we show that $(A,\phi)$ is the 1-product of the $(A_i,\phi_i)$. Suppose that $(B,\psi)$ is a PCA over $\mathcal{D}$ and that, for each $i\in I$, we have an applicative morphism $(q_i,f_i)\colon (B,\psi)\to (A_i,\phi_i)$. The $q_i$ have a unique amalgamation $q\colon \mathcal{D}\to \mathcal{C}$ such that $p_i\circ q = q_i$ for all $i\in I$. Moreover, there exists a unique relation $f\subseteq q(B)\times A$ such that $p_i(f)\subseteq q_i(B)\times A_i$ is equal to $f_i$, and it is easily seen that $(q,f)$ is in fact an applicative morphism $(B,\psi)\to (A,\phi)$.

Now suppose we have two applicative morpisms $(q,f),(r,g)\colon (B,\psi)\to(A,\phi)$, and for each $i$, an applicative transformation $\mu_i\colon (p_i,\delta_{A_i})\circ(q,f)\To (p_i,\delta_{A_i})\circ(r,g)$. Write $q_i = p_i\circ q$ and $f_i = p_i(f)$, so that $(p_i,\delta_{A_i})\circ(q,f) = (q_i,f_i)$, and similarly for $(r,g)$. Then the fact that $\mu_i$ is an applicative transformation tells us that $f_i\leq g_i\circ\overline{\mu_{i,A_i}}$. The $\mu_i$ have a unique amalgamation $\mu\colon q\To r$ such that $p_i\mu = \mu_i$ for all $i$. If $U_i\in\phi_i$ realizes this inequality, then $(U_i)_{i\in I}\in\phi$ realizes $f\leq g\circ\overline{\mu_A}$, so $\mu$ is an applicative transformation $(q,f)\To(r,g)$.
\end{proof}
\begin{rem}\label{remarks_product}
\begin{itemize}
\item[(i)]	Observe that the proof above uses the Axiom of Choice on the index set $I$.
\item[(ii)]	If all the $(A_i,\phi_i)$ are absolute PCAs, then so is $\prod_{i\in I} (A_i,\phi_i)$. \ruitje
\end{itemize}
\end{rem}

The fibers $\mathsf{PCA}_\mathcal{C}$ all have finite pseudocoproducts. We will show that these are preserved by reindexing, but first we need the following auxilliary result.
\begin{lem}\label{product_generate}
Let $p\colon \mathcal{C}\to\mathcal{D}$ be regular, let $A$ and $B$ be PASs over $\mathcal{C}$, and suppose that $G\subseteq\mathcal{P}^\ast A$ and $H\subseteq \mathcal{P}^\ast B$ are combinatorially complete. Then $\langle\langle G\rangle\times \langle H\rangle\rangle = \langle G\times H\rangle$ as filters on the PAS $A\times B$.
\end{lem}
\begin{proof}
We clearly have $G\times H\subseteq \langle G\rangle\times \langle H\rangle\subseteq \langle\langle G\rangle\times \langle H\rangle\rangle$ and since the right-hand side is a filter, this yields $\langle G\times H\rangle\subseteq \langle\langle G\rangle\times \langle H\rangle\rangle$.

For the converse inclusion, let $U\times V\in \langle G\rangle\times \langle H\rangle$. Then there exist terms $t(\vec{x})$ and $s(\vec{y})$, and elements $\vec{W}\in G$ and $\vec{T}\in H$ such that $t(\vec{W})\denotes$, $t(\vec{W})\subseteq U$, $s(\vec{T})\denotes$ and $s(\vec{T})\subseteq V$. By our assumption, we can choose combinators $\mathsf{K},\mathsf{I}\in G$ and $\overline{\mathsf{K}},\mathsf{I}\in H$. For every component $W_i$ from $\vec{W}$, define $W'_i = W_i\times\mathsf{I}\in G\times H$. Moreover, it is easily see that $t(\vec{W}')$ is defined and equal to $t(\vec{W})\times \mathsf{I}$. Similarly, set $T_j'= \mathsf{I}\times T_j\in G\times H$, so that $s(\vec{T}')$ is defined and equal to $\mathsf{I}\times s(\vec{T})$. Now define the term $r(\vec{x},\vec{y},z)$ as $z\cdot t(\vec{x})\cdot s(\vec{y})$. Then $r(\vec{W}',\vec{T}',\mathsf{K}\times\overline{\mathsf{K}})$ is defined and equal to
\begin{align*}
(\mathsf{K}\times\overline{\mathsf{K}})\cdot t(\vec{W}')\cdot s(\vec{T}') &= (\mathsf{K}\times\overline{\mathsf{K}})\cdot (t(\vec{W})\times \mathsf{I})\cdot (\mathsf{I}\times s(\vec{T}))\\
&= (\mathsf{K}\cdot t(\vec{W})\cdot\mathsf{I}) \times(\overline{\mathsf{K}}\cdot\mathsf{I}\cdot s(\vec{T}))\\
&= t(\vec{W})\times s(\vec{T})\\
&\subseteq U\times V.
\end{align*}
Since $\vec{W}',\vec{T}',\mathsf{K}\times\overline{\mathsf{K}}\in G\times H$, we can conclude that $U\times V\in \langle G\times H\rangle$, so we have shown that $\langle G\rangle\times \langle H\rangle\subseteq \langle G\times H\rangle$. Since the right-hand side is a filter, this yields $\langle\langle G\rangle\times \langle H\rangle\rangle\subseteq \langle G\times H\rangle$, as desired.
\end{proof}
\begin{prop}
If $p\colon\mathcal{C}\to\mathcal{D}$ is regular, then $p^\ast \colon \mathsf{PCA}_\mathcal{C}\to\mathsf{PCA}_\mathcal{D}$ preserves finite pseudocoproducts.
\end{prop}
\begin{proof}
The only difficult part is showing that, for PCAs $(A,\phi)$ and $(B,\psi)$ over $\mathcal{C}$, the filters on $p^\ast((A,\phi)+(B,\psi))$ and $p^\ast(A,\phi)+p^\ast(B,\psi)$ coincide (modulo the isomorphism $p(A\times B)\simeq p(A)\times p(B)$). It is clear that $p(\phi\times\psi) = p(\phi)\times p(\psi)$, so by combining \cref{HF_on_objects}(ii) and \cref{product_generate}, we see that
\[\langle p(\langle \phi\times \psi\rangle)\rangle = \langle p(\phi\times\psi)\rangle = \langle p(\phi)\times p(\psi)\rangle = \langle\langle p(\phi)\rangle\times \langle p(\psi)\rangle\rangle,\]
as desired.
\end{proof}

The category $\mathsf{REG}$ also has pseudocoproducts: if $\mathcal{C}$ and $\mathcal{D}$ are regular, then their pseudocoproduct is $\mathcal{C}\times \mathcal{D}$. The inclusions $\mathcal{C}\stackrel{\iota_0}{\longrightarrow}\mathcal{C}\times\mathcal{D}\stackrel{\iota_1}{\longleftarrow}$ are given by $\iota_0(X) = (X,1)$ and $\iota_1(Y) = (1,Y)$. If the forgetful functor $\mathsf{PCA}\to\mathsf{REG}$ were a 2-opfibration, then this would mean we also have pseudocoproducts in $\mathsf{PCA}$. However, we have not shown this, since we do not in general have a natural transformation $\mu^\ast\colon p^\ast\To q^\ast$ when $\mu\colon p\To q$. This means that, while it does have 2-cocartesian lifts of 1-cells, the forgetful functor $\mathsf{PCA}\to\mathsf{REG}$ lacks cartesian lifts of 2-cells. We do, on the other hand, have cartesian lifts of \emph{invertible} 2-cells by \cref{transporting_mu}(iii), which means we can get a partial result.

For a 2-category $\mathcal{B}$, let $\mathcal{B}_\text{iso}$ be the 2-category which has the same 0- and 1-cells as $\mathcal{B}$, but whose 2-cells are only the invertible 2-cells of $\mathcal{B}$. Then we also have a forgetful functor $\mathsf{PCA}_\text{iso}\to\mathsf{REG}_\text{iso}$, which is a 2-opfibration. We know that $\mathsf{REG}_\text{iso}$ still has finite pseudocoproducts, as do all the fibers $(\mathsf{PCA}_\mathcal{C})_\text{iso}$, and reindexing still preserves finite pseudocoproducts. From this, we can conclude that $\mathsf{PCA}_\text{iso}$ also has finite pseudocoproducts. In the binary case, they are computed as follows. If $(A,\phi)$ and $(B,\psi)$ are PCAs over $\mathcal{C}$ and $\mathcal{D}$ respectively, then their pseudocoproduct in $\mathsf{PCA}_\text{iso}$ is formed by taking the pseudocoproduct $\iota_0^\ast(A,\phi)+\iota_1^\ast(B,\psi)$ in $\mathcal{C}\times\mathcal{D}$. Its onderlying object is $(A,1)\times (1,B)\simeq (A,B)$, and its filter is generated by
\[\{(U,V)\subseteq (A,B)\mid U\in\phi, V\in\psi\}.\]
But this set is already a filter; and moreover, it is the filter that yields the 2-product of $(A,\phi)$ and $(B,\psi)$ in $\mathsf{PCA}$, and hence also in $\mathsf{PCA}_\text{iso}$. So we see that the pseudocoproduct of $(A,\phi)$ and $(B,\psi)$ in in $\mathsf{PCA}_\text{iso}$ coincides with their product. It is also easy to see that the pseudoinitial object of $\mathsf{PCA}_\text{iso}$ coincides with the terminal object, so $\mathsf{PCA}_\text{iso}$ has a pseudozero object. Moreover, one readily calculates that the composition
\[(A,\phi)\to (A,\phi)\times (B,\psi) \to (A,\phi)\]
is the identity on $(A,\phi)$, whereas
\[(A,\phi)\to (A,\phi)\times (B,\psi) \to (B,\psi)\]
is a zero morphism. We can conclude:
\begin{prop}
The 2-category $\mathsf{PCA}_\textup{iso}$ has finite pseudobiproducts.
\end{prop}

\section{The Functor $\asm$}\label{section_asm}

In the previous section, we defined applicative morphisms and transformations. At this point, it is not so clear how well-behaved these notions are. In this section, we investigate the compatibility between these notions and the category of assemblies defined in \cref{first_section}. More precisely, we extend $\asm$ to a 2-functor $\mathsf{PCA}\to\mathsf{Cat}$ and we characterize its image on the 1- and 2-cells. The proofs in this section are adaptations of the proofs for the corresponding results in the case of classical PCAs. In this sense, the results in this section are not innovative. Neventheless, we include them because they provide evidence that our notions of applicative morphism and transformation are the `right' ones.

We start by extending the assignment $\asm$.
\begin{defi}
\begin{itemize}
\item[(i)]
Let $(p,f)\colon (A,\phi)\to (B,\psi)$ be an applicative morphism. We define $\asm(p,f)$ as the functor $F\colon \asm(A,\phi)\to \asm(B,\psi)$ determined by:
\begin{itemize}
\item	$|FX| = p(|X|)$;
\item	$E_{FX} = f\circ p(E_X) = \{(x,b)\in p(|X|)\times B\mid \exists a: p(A)\s (p(E_X)(x,a)\wedge f(a,b))\}$;
\item	$Fg = p(g)$.
\end{itemize}
\item[(ii)]
If $\mu\colon (p,f)\To(q,g)$ is an applicative transformation, then define $\asm(\mu)\colon \asm(p,f)\To\asm(q,g)$ by $\asm(\mu)_X = \mu_{|X|}$ for $X\in\asm(A,\phi)$. \ruitje
\end{itemize}
\end{defi}
\begin{prop}
The assignment $\asm$ is a well-defined 2-functor $\mathsf{PCA}\to\mathsf{Cat}$.
\end{prop}
\begin{proof}
Let $(p,f)\colon (A,\phi)\to (B,\psi)$ be an applicative morphism, and write $F$ for $\asm(p,f)$. We first show that $F$ is indeed a functor, for which we only need to check that $p(k)$ is actually a morphism $FX\to FY$ when $k\colon X\to Y$ is a morphism in $\asm(A,\phi)$. If $U\in\phi$ tracks $k$ and $V\in\psi$ tracks $f$, then any realizer of $\lambda x. V\cdot f(p(U))\cdot x$ tracks $p(k)\colon FX\to FY$, which shows that $F$ is indeed a functor.

Now suppose that $(q,g)\colon (B,\psi)\to(C,\chi)$ is another applicative morphism. Since $q$ is a regular functor, we have that
\[g\circ q(f\circ p(E_X)) = g\circ q(f) \circ q(p(E_X)),\]
for every assembly $X$ over $(A,\phi)$. From this, it easily follows that $\asm(q,g)\circ \asm(p,f) = \asm((q,g)\circ(p,f))$, so $\asm$ is a 1-functor.

Now suppose that $(p,f),(q,g)\colon (A,\phi)\to (B,\psi)$ and $\mu\colon (p,f)\To (q,f)$. The naturality of $\asm(\mu)$ is obvious, but we need to check that $\asm(\mu)_X = \mu_{|X|}$ is actually a morphism $\asm(p,f)(X)\to \asm(q,g)(X)$. Applying the naturality of $\mu$ to the inclusion $E_X\subseteq |X|\times A$ yields that
\[p(E_X)(x,a) \models_{x:p(|X|); a:p(A)} q(E_X)(\mu_{|X|}(x),\mu_A(a))\]
is valid in the underlying category of $(B,\psi)$. From this, it easily follows that a tracker of the applicative transformation $\mu$ also tracks $\mu_{|X|}$, so the latter is indeed a morphism. It is clear from the definitions that $\asm$ respects the vertical and horizontal composition of 2-cells.
\end{proof}
Since $p\circ\Gamma = \Gamma\circ\asm(p,f)$, we immediately get the following corollary.
\begin{cor}
Let $(A,\phi)$ and $(B,\psi)$ be PCAs over $\mathcal{C}$ and $\mathcal{D}$ respectively. If $(A,\phi)$ and $(B,\psi)$ are isomorphic \textup{(}equivalent\textup{)}, then the functors $\Gamma\colon\asm(A,\phi)\to\mathcal{C}$ and $\Gamma\colon\asm(B,\psi)\to\mathcal{D}$ are also isomorphic \textup{(}equivalent\textup{)}.
\end{cor}
Combining this with \cref{absoluteness}, we find that absoluteness is stable under equivalence of PCAs.

The category of assemblies for a classical (relative) PCA is always a quasitopos. This cannot be guaranteed in our setting, since not all the relevant constructions can be carried out in the base category $\mathcal{C}$. We do still have that $\asm(A,\phi)$ is itself a regular category. In the proof of the next proposition, we describe the regular structure of $\asm(A,\phi)$, which we will need later in this section.

\begin{prop}\label{asm_reg}
If $(A,\phi)$ is a PCA over $\mathcal{C}$, then $\asm(A,\phi)$ is a regular category, and the functors $\Gamma\colon \asm(A,\phi)\to\mathcal{C}$ and $\nabla\colon\mathcal{C}\to\asm(A,\phi)$ are regular.
\end{prop}
\begin{proof}
As we have already remarked, the constant object $\nabla 1$ is a terminal object in $\asm(A,\phi)$. If $X$ and $Y$ are assemblies, then we can form their product by taking $|X\times Y| = |X|\times |Y|$ and
\[E_{X\times Y} = \{(x,y,a)\in|X|\times|Y|\times A\mid \exists b,c,d: A\s (\mathsf{P}(b)\wedge E_X(x,c)\wedge E_Y(y,d)\wedge bcd=a)\}.\]
The projections of $|X\times Y|$ onto $|X|$ and $|Y|$ are morphisms of assemblies, since they are tracked by $\mathsf{P}_0$ and $\mathsf{P}_1$ respectively. If $f\colon Z\to X$ and $g\colon Z\to Y$ are morphisms, then there exists a unique mediating arrow $\langle f,g\rangle\colon |Z|\to |X|\times |Y|$ in $\mathcal{C}$. This is also a morphism $Z\to X\times Y$ since a realizer of $\lambda x.\mathsf{P}(Ux)(Vx)$ tracks $\langle f,g\rangle$, where $U,V\in\phi$ track $f$ and $g$ respectively. Observe that this construction of the product works for any choice of $\mathsf{P}\in\phi$ realizing $\lambda xyz. zxy$, a fact we will use later in this section.

If $f,g\colon X\to Y$ is a parallel pair, then first form the equalizer $m\colon |Z|\mono |X|$ of $f$ and $g$ in $\mathcal{C}$. Now define
\[E_Z = (m\times\id_A)^{-1}(E_X) = \{(z,a)\in |Z|\times A\mid E_X(m(z),a)\}.\]
Then $m$ is a morphism $Z\to X$, since it is tracked by $\mathsf{I}$. Moreover, if $h\colon W\to X$ satisfies $fh=gh$, then we get a unique mediating arrow $k\colon |W|\to |Z|$, which is a morphism since every tracker of $h$ is also a tracker of $k$.

It follows that the pullback of a cospan $X\stackrel{f}{\longrightarrow} Z\stackrel{g}{\longleftarrow} Y$ is formed as follows. First take the pullback
\[\begin{tikzcd}
{|W|} \arrow[r, "p"] \arrow[d, "q"'] \arrow[rd,phantom,"\text{\Large{$\lrcorner$}}", pos=0] & {|X|} \arrow[d, "f"] \\
{|Y|} \arrow[r, "g"']                & {|Z|}               
\end{tikzcd}\]
in $\mathcal{C}$, and then equip $|W|$ with
\[E_W = \{(w,a)\in |W|\times A\mid \exists b,c,d:A\s (\mathsf{P}(b)\wedge E_X(p(w),c)\wedge E_Y(q(w),d)\wedge bcd=a)\}.\]
From the description of finite limits above, it follows that $\Gamma$ preserves all finite limits. In particular, $\Gamma$ preserves monos. Being faithful, $\Gamma$ also reflects monos.

Now suppose that $e\colon X\to Y$ is a regular epimorphism in $\asm(A,\phi)$. Since $\Gamma$ preserves finite limits and colimits, it follows that $e\colon |X|\to |Y|$ must be a regular epimorphism in $\mathcal{C}$. Suppose, conversely, that we have an arrow $e\colon X\to Y$ such that $e\colon |X|\to |Y|$ is a regular epimorphism. Then we can form the coequalizer $e'\colon X\to Y'$ of the kernel pair of $e$ by taking $|Y'| = |Y|$, $e'=e$ and 
\[E_{Y'} = \{(y,a)\in |Y|\times A\mid \exists x:|X|\s (e(x)=y\wedge E_X(x,a))\}.\]
It follows immediately that $e\colon X\to Y$ is a regular epimorphism if and only if $\id_{|Y|}$ is a morphism $Y\to Y'$. And this is to say that there exists a $U\in\phi$ satisfying
\[E_Y(y,a)\wedge U(r) \models_{y:|Y|; r,a:A} ra\denotes\wedge\ \exists x:|X|\s (e(x)=y\wedge E_X(x,ra)).\]
We will say that such a $U$ witnesses the fact the $e$ is a regular epimorphism. If there exists such a $U\in\phi$, then in particular, $U$ is inhabited and in this case, it follows easily by soundness that $e\colon |X|\to|Y|$ is a regular epimorphism. Therefore, we can summarize the above as follows: $e\colon X\to Y$ is a regular epimorphism if and only if there exists a $U\in\phi$ witnessing this.

It follows at once that $\asm(A,\phi)$ has regular epi-mono factorizations. Indeed, if $f\colon X\to Y$ is a morphism, then we can first factor $f$ in $\mathcal{C}$ as $|X|\stackrel{e}{\twoheadrightarrow} |Z|\stackrel{m}{\mono} |Y|$, and then put
\[E_Z = \{(z,a)\in |Z|\times A\mid \exists x: |X|\s (e(x)=z\wedge E_X(x,a))\}.\]
Here $\mathsf{I}$ both tracks $e\colon X\to Z$ and witnesses the fact that $e$ is a regular epimorphism, and any tracker of $f$ also tracks $m\colon Z\to Y$.

So, in order to prove that $\asm(A,\phi)$ is regular, it remains to show that regular epimorphisms are stable under pullback. Suppose that we have a cospan $X\stackrel{f}{\longrightarrow} Z\stackrel{g}{\longleftarrow} Y$ with $f$ a regular epimorphism, and define its pullback $W$ as above. Let $U\in\phi$ witness the fact that $f$ is a regular epimorphism and let $V\in \phi$ track $g$. Then there exists a $T\in\phi$ realizing
\[\lambda x.\mathsf{P}(U(Vx))x.\]
We claim that $T$ also witnesses the fact that $q$ is a regular epimorphism. Let $T'$ be a realizer of $\lambda yzwx. y(z(wx))x$ such that $T\subseteq T'\mathsf{P}UV$.

Now reason inside $\mathcal{C}$ and consider $y\in |Y|$ and $t,a\in A$ such that $E_Y(y,a)$ and $T(t)$. Then there exist $t',b,r,s\in A$ such that $T'(t')$, $\mathsf{P}(b)$, $U(r)$, $V(s)$ and $t = t'brs$. Since $E_Y(y,a)$ and $V(s)$, we see that $sa\denotes$ and $E_Z(g(y),sa)$. Combining this with $U(r)$ yields that $r(sa)\denotes$, and we find an $x\in |X|$ such that $f(x)=g(y)$ and $E_X(x,r(sa))$. Since $|W|$ is the pullback of $f$ and $g$, we get that there is a $w\in |W|$ such that $p(w) = x$ and $q(w)=y$. Since $\mathsf{P}(b)$, $E_X(p(w),r(sa))$ and $E_Y(q(w),a)$, it follows that $E_W(w,b(r(sa))a)$. We conclude that $ta$ is defined and equal to $t'brsa = b(r(sa))a$, so $E_W(w,ta)$, as desired.

If $e\colon X\twoheadrightarrow Y$ is a regular epimorphism in $\mathcal{C}$, then we easily see that $\mathsf{I}$ witnesses the fact that $\nabla e$ is a regular epimorphism, so $\nabla$ preserves regular epimorphisms. Since $\Gamma\dashv\nabla$, we conclude that $\Gamma$ and $\nabla$ are both regular.
\end{proof}

We will say that an assembly is a \emph{constant object} if it is isomorphic to an object in the image of $\nabla$. As a more direct characterization, we have: an assembly $X$ is a constant object if and only if there exists a $U\in\phi$ such that $|X|\times U\subseteq E_X$. 

Every assembly is a subobject of a constant object. Indeed, if $\eta$ is the unit of the adjunction $\Gamma\dashv \nabla$, then $\eta_X\colon X\to\nabla\Gamma X$ is mono, since its underlying arrow is simply $\id_{|X|}$. We introduce the following definition from \cite{wps}.
\begin{defi}
A morphism of assemblies $f\colon X\to Y$ is called \emph{prone} if the naturality square
\[\begin{tikzcd}[column sep=large]
X \arrow[r, "f"] \arrow[d, "\eta_X"']      & Y \arrow[d, "\eta_Y"] \\
\nabla\Gamma X \arrow[r, "\nabla\Gamma f"'] & \nabla\Gamma Y       
\end{tikzcd}\]
is a pullback square. \ruitje
\end{defi}

A morphism $f\colon X\to Y$ is prone if and only if the identity on $|X|$ is an isomorphism $X\to X'$, where $|X'| = |X|$ and $E_{X'} = (f\times\id_A)^\ast(E_Y)$. In particular, if $Y$ is an assembly and $m\colon U\mono|Y|$ is a subobject, then there exists a unique prone subobject of $Y$ whose underlying subobject in $\mathcal{C}$ is $m$, given by $|X| = U$ and $E_X = (m\times \id_A)^\ast(E_Y) = E_Y|_{U\times A}$. Observe that every regular subobject is also prone (but conversely, a prone subobject $m\colon X\mono Y$ is regular if and only if $m$ is already regular in $\mathcal{C}$). We also remark that prone morphisms are stable under pullback. One can prove this directly from the definition, but also from the alternative description of prone morphisms above, along with the explicit description of a pullback in $\asm(A,\phi)$ provided in the proof of \cref{asm_reg}.

We finish this section characterizing the functors and natural transformations that are of the form $\asm(p,f)$ and $\asm(\mu)$ respectively. Our proofs are an adaptation of Longley's methods from \cite{longleyphd}, who obtained similar results for classical absolute PCAs.
\begin{thm}\label{image_asm}
Let $(A,\phi)$ and $(B,\psi)$ be PCAs over $\mathcal{C}$ and $\mathcal{D}$ respectively.
\begin{itemize}
\item[\textup{(}i\textup{)}]	If $p\colon\mathcal{C}\to \mathcal{D}$ is a regular functor, then a functor $F\colon \asm(A,\phi)\to\asm(B,\psi)$ is naturally isomorphic to a functor the form $\asm(p,f)$ for some applicative morphism $(p,f)$ if and only if:
\begin{itemize}
\item	$F$ is regular;
\item	$\Gamma \circ F \simeq p\circ\Gamma$;
\item	$F\circ\nabla \simeq \nabla\circ p$.
\end{itemize}
\item[\textup{(}ii\textup{)}]	Suppose that $(p,f),(q,g)\colon (A,\phi)\to(B,\psi)$ are applicative morphisms and that $\mu\colon p\To q$ is natural transformation. For an assembly $X\in\asm(A,\phi)$, define $\tilde{\mu}_X = \mu_{|X|}$. Then $\tilde{\mu}$ is a natural transformation $\asm(p,f)\To\asm(q,g)$ if and only $\mu$ is an applicative transformation $(p,f)\To(q,g)$.
\end{itemize}
\end{thm}
\begin{proof}
(i) First suppose that $F$ is of the form $\asm(p,\id)$, i.e. that our applicative morphism is cocartesian. We will show that $F$ is regular. For binary products, one needs to observe the following: if $X,Y\in\asm(A,\phi)$, then $|F(X\times Y)| = p(|X|)\times p(|Y|)$ and $E_{F(X\times Y)}$ is the subobject
\[\{(x,y,a)\mid \exists b,c,d: p(A)\s (p(\mathsf{P})(b)\wedge p(E_X)(x,c)\wedge p(E_Y)(y,d)\wedge bcd=a)\}\]
of $p(|X|)\times p(|Y|)\times p(A)$. Since $p(\mathsf{P})$ belongs to $\langle p(\phi)\rangle$ and realizes $\lambda xyz.zxy$ w.r.t.\@ $p(A)$, we can conclude that $F(X\times Y)$ is a product of $FX$ and $FY$. For regular epimorphisms, we observe: if $U\in\phi$ witnesses the fact that $e\colon X\to Y$ is a regular epimorphism, then $p(U)\in\langle p(\phi)\rangle$ witnesses this for $p(e)=Fe\colon FX\to FY$. The preservation of the terminal object and of equalizers is easy, so we conclude that $F$ is indeed regular. Moreover, if $F$ is of the form $\asm(\id,f)$ for some vertical applicative morphism $f$, then we can show that $F$ is regular by adapting Longley's argument from \cite{longleyphd} (Propositions 2.2.2 and 2.2.3). 

Now take a general $F$ of the form $\asm(p,f)$, with $(p,f)\colon (A,\phi)\to(B,\psi)$. Since any applicative morphism can be decomposed as a cocartesian morphism followed by a vertical one, the above implies that $F$ is regular. Moreover, we have $\Gamma \circ F = p\circ\Gamma$ by definition. Finally, suppose that $X\in\asm(A,\phi)$ is constant, and take $U\in\phi$ such that $|X|\times U\subseteq E_X$. Then it easily follows that $p(|X|)\times f(p(U))\subseteq p(E_X)\circ f = E_{FX}$. We have $p(U)\in p(\phi)$, so $f(p(U))\in\psi$, and therefore $FX$ is a constant object as well. In other words, $F$ preserves constant objects, which implies $F\circ\nabla \simeq \nabla\circ p$.

For the converse, suppose that $F$ satisfies the three requirements above. It is easy to see that $F$ is naturally isomorphic to an $F'$ such that $\Gamma \circ F' = p\circ\Gamma$. Therefore, we assume that we have $\Gamma \circ F = p\circ\Gamma$ on the nose. Since $F$ preserves pullbacks and commutes with $\nabla$ and $\Gamma$, we see that $F$ also preserves prone morphisms. In particular, $F$ preserves prone subobjects. 

First, we define the assembly $R\in \asm(A,\phi)$ by $|R| = A$ and $E_R = \delta_A$. (This object is also known as the \emph{object of realizers}.) Then $FR$ satisfies $|FR| = p(|R|) = p(A)$, so $E_{FR}$ is a subobject of $p(A)\times B$. We claim that $f:= E_{FR}$ is an applicative morphism $p^\ast(A,\phi)\to(B,\psi)$, so that $(p,f)$ is an arrow $(A,\phi)\to (B,\psi)$ in $\mathsf{PCA}$.

Requirement (a) is immediate since $FR$ is an assembly. For requirement (b), consider the prone subobject $S\mono R\times R$ with $|S| = D \subseteq A\times A$. If $k\colon D\to A$ is the application map, then $k$ is a morphism $S\to R$, since it is tracked by any realizer of $\lambda x. (\mathsf{P}_0x)(\mathsf{P}_1x)$. Since $F$ preserves products and prone subobjects, we see that $FS\mono FR\times FR$ is also prone, and that $p(k) = F(k)$ is a morphism $FS\to FR$, say tracked by $U\in\psi$. Without loss of generality,
\[E_{FS} = \{(a,a',b)\in p(D)\times B\mid \exists c,d,d': B\s (\mathsf{P}(c)\wedge E_{FR}(a,d)\wedge E_{FR}(a',d')\wedge cdd'=b)\}.\]
From this it easily follows that any realizer of $\lambda xy.U(\mathsf{P}xy)$ tracks $f$, so requirement (b) holds. For requirement (c), consider a $U\in\phi$. Define the prone subobject $R_U\mono R$ by $|R_U| = U$. Then $e\colon R_U\to 1$ is a regular epimorphism, since $\mathsf{K}U$ witnesses this fact. It follows that $FR_U\mono FR$ is prone and $Fe\colon FR_U\to F1$ is regular epi, which means that without loss of generality,
\[E_{F1} = \{b\in B\mid \exists a: p(U) E_{FR}(a,b)\} = f(p(U)).\]
But $F1$, being the terminal object of $\asm(B,\psi)$, must be a constant object, which implies that $f(p(U))\in\psi$. By using \cref{lift_left}, we see that requirement (c) holds as well.

It remains to show that $F\simeq \asm(p,f)$. Let $X\in\asm(A,\phi)$, and consider the assembly $\nabla|X|\times R$. We can describe this as the assembly such that $|\nabla |X|\times R| = |X|\times A$ and 
\[E_{\nabla|X|\times R} = |X|\times \delta_A = \{(x,a,a')\in (|X|\times A)\times A\mid a=a'\}.\]
Let $Y\mono \nabla|X|\times R$ be the prone subobject with $|Y| = E_X\subseteq |X|\times A$. Then the projection map $\pi\colon E_X\to |X|$ is a regular epimorphism $Y\to X$, since $\mathsf{I}$ witnesses this fact. It follows that the subobject $FY\mono F\nabla|X|\times FR$ is prone and that $F\pi\colon FY\to FX$ is regular epi. Since $F\nabla|X|$ is constant, we can assume without loss of generality that
\[E_{F\nabla|X|\times FR} = p(|X|)\times E_{FR} = \{(x,a,b)\in (p(|X|)\times p(A))\times B\mid E_{FR}(a,b)\}.\]
We see that the identity on $p(|X|)$ is an isomorphism $FX\to X'$, where $|X'| = p(|X|)$ and
\[E_{X'} = \{(x,b)\in p(|X|)\times B\mid \exists a: p(A)\s (p(|Y|)(x,a)\wedge E_{FR}(a,b))\}.\]
Since $|Y| = E_X$ and $E_{FR}=f$, this means that $X' = \asm(p,f)(X)$, which completes the proof of (i).

For (ii), we already know that $\tilde{\mu} = \asm(\mu)$ is a natural transformation if $\mu$ is an applicative transformation. Conversely, suppose that $\tilde{\mu}$ is a natural transformation. Then in particular, $\mu_A$ should be a morphism $\asm(p,f)(R)\to\asm(q,g)(R)$, which means exactly that $\mu$ is an applicative transformation $(p,f)\To (g,q)$.
\end{proof}
\begin{rem}
Let $(A,\phi)$ and $(B,\psi)$ be PCAs over $\mathcal{C}$ and $\mathcal{D}$ respectively. Then we can view $\asm$ as a functor $\mathsf{PCA}((A,\phi),(B,\psi))\to \mathsf{REG}(\asm(A,\phi),\asm(B,\psi))$. This functor is faithful, and \cref{image_asm} describes its essential image. \ruitje
\end{rem}
\begin{rem}
In the case of classical absolute PCAs, the requirement that $F$ preserves constant objects can be removed, since any functor commuting with $\Gamma$ automatically commutes with $\nabla$ (\cite{longleyphd}, Proposition 2.3.3). The argument used to show this is set theoretic and involves selecting a set of sufficiently high cardinality. We have not been able to generalize this proof to the present situation. \ruitje
\end{rem}

\section{Products and Slicing}\label{section_slice}

In this section, we show that categories of the form $\asm(A,\phi)$ are, up to equivalence, closed under small products and under slicing. The first claim, which in fact holds up to isomorphism, follows from the following proposition, whose proof is omitted.

\begin{prop}
The functor $\asm$ preserves small \textup{(}2-\textup{)}products. \qed
\end{prop}

We now turn our attention to showing that a slice of a category of the form $\asm(A,\phi)$ is again of this form. This result was already obtained by Stekelenburg (\cite{wps}, p.\@ 62), but in an indirect way. We will construct the required PCA explicitly, and use this to compute some specific slices of categories of assemblies.

But first, let us explain how the obstacle mentioned in the Introduction is removed in the current setting. There we explained that the terminal object of a category of the form $\asm(A)$ (or $\asm(A,C)$) is always projective, whereas the terminal object of $\asm(A)/I$ is only projective if $I$ itself is projective in $\asm(A)$. It turns out that for PCAs in the current setting, the terminal object of $\asm(A,\phi)$ need not be projective. An obvious reason for this is that $\Gamma\colon\asm(A,\phi)\to\mathcal{C}$ preserves projectives (this easily follows from the fact that $\nabla$ is regular), so the terminal object of $\asm(A,\phi)$ can only be projective if the terminal object of $\mathcal{C}$ is already projective. But even if $1\in\mathcal{C}$ is projective, the projectivity of $1\in\asm(A,\phi)$ puts a very strong requirement on the filter $\phi$. We prove this in the following proposition, the first part of which is Proposition 2.5.9 from \cite{wps}.

\begin{prop}
Let $(A,\phi)$ be a PCA over the regular category $\mathcal{C}$.
\begin{itemize}
\item[\textup{(}i\textup{)}]	If $1\in\asm(A,\phi)$ is projective, then $\phi$ is generated by singletons.
\item[\textup{(}ii\textup{)}]	If $1\in\mathcal{C}$ is projective and $\phi$ is generated by singletons, then $1\in\asm(A,\phi)$ is projective.
\end{itemize}
\end{prop}
\begin{proof}
(i) Let $R\in\asm(A,\phi)$ be the object of realizers, i.e., $|R|=A$ and $E_R = \delta_A\subseteq A\times A$, and consider the set
\[C=\{a\colon 1\mono A\mid a\mbox{ is a morphism }1\mono R\mbox{ in }\asm(A,\phi)\}\]
of global sections of $A$. We claim that $C$ generates $\phi$. To show that $C$ is closed under application, consider again the prone subobject $S\subseteq R\times R$ given by $|S|=D$. As we saw in the previous section, the application map $D\to A$ is a morphism $S\to R$. Now suppose we have $a,b\in C$ such that $ab\denotes$. Then the global section $\langle a,b\rangle\colon 1\mono A\times A$ factors through $D$, and we get the diagram
\[\begin{tikzcd}
          & 1 \arrow[d, hook] \arrow[ld, "{\langle a,b\rangle}"', hook] \arrow[rd, "ab", hook] &   \\
R\times R & S \arrow[l, hook] \arrow[r, "\cdot"']                                              & R
\end{tikzcd}\]
in $\asm(A,\phi)$, which shows that $ab\in C$ as well. It remains to show that $\phi = \phi_C$.

Let $a\in C$, and let $U\in\phi$ track the morphism $a\colon 1\mono R$. Then $UA\denotes$ and $UA\subseteq c$, which shows that $c\in \phi$ as well. From this, we may conclude that $\phi_C = \langle C\rangle \subseteq\phi$.

Conversely, suppose that $U\in\phi$, and consider the prone subobject $R_U\mono R$ with $|R_U|=U$. Then $R_U\to 1$ is a regular epimorphism, for $\mathsf{K}U$ witnesses this. Since we assumed that $1\in\asm(A,\phi)$ is projective, this regular epimorphism splits and yields a global section $a\colon 1\mono R_U\mono R$. We see that $a\in C$, and $a\subseteq U$, so $U\in\phi_C$, as desired.

(ii) Suppose that $\phi=\phi_C$ for some set $C\subseteq\mathcal{C}(1,A)$ that is closed under application. Then $C$ is nonempty, so we can select a $c_0\in C$. Now let $1'\in\asm(A,\phi)$ be the object defined by $|1'| = 1$ and $E_{1'} = 1 \stackrel{c_0}{\mono} A \cong 1\times A$. Then $1'$ is isomorphic to $1=\nabla 1$, since $\mathsf{I}\in\phi$ tracks $\id_1$ as a morphism $1'\to 1$, whereas $\mathsf{K}c_0\in\phi$ tracks $\id_1$ as a morphism $1\to 1'$. This means that $1'$ is also a terminal object of $\asm(A,\phi)$, so it suffices to prove that $1'$ is projective.

To this end, let $X\to 1'$ be a regular epimorphism, and let $U\in\phi$ witness this fact. Then we can select a $c\in C$ such that $c\subseteq U$, and we find that $\models \exists x:|X|\s (E_X(x,cc_0))$. Since $1\in\mathcal{C}$ is projective, this implies that there exists a global section $x\colon 1\mono |X|$ such that $\models E_X(x,cc_0)$. Moreover, $c\in\phi$ tracks $x$ as a morphism $1'\mono X$, which completes the proof that $1'$ is projective.
\end{proof}

For further results on the behavior of projective objects of $\asm(A,\phi)$ in the case where $\phi$ is generated by singletons, we refer the reader to Section 2.5 of \cite{wps}.

Now let $(A,\phi)$ be a PCA over $\mathcal{C}$ and let $I$ be an object of $\mathcal{C}$. We consider the regular category $\mathcal{C}/I$ and the pullback functor $I^\ast\colon\mathcal{C}\to\mathcal{C}/I$, which is a regular functor. It immediately follows that $I^\ast(A)$ can be equipped with a partial applicative structure in the obvious way. First of all, we will spell out what a PCA $(I^\ast(A),\psi)$ over $\mathcal{C}/I$ is in terms of the internal logic of $\mathcal{C}$.

Since an arrow is mono in $\mathcal{C}/I$ if and only if it is mono in $\mathcal{C}$, we see that the subobjects of $I^\ast(A)$ in $\mathcal{C}/I$ are the subobjects of $I\times A$ in $\mathcal{C}$. A subobject $U\subseteq I\times A$ is inhabited in $\mathcal{C}/I$ if and only if it is fiberwise inhabited in $\mathcal{C}$, i.e., $\models_{i:I} \exists a: A\ (U(i,a))$. The composition of two subobjects $U,V\subseteq I\times A$ is defined iff $U\times_I V\subseteq I\times D$, and in this case, $UV$ is the image of the map $U\times_I V\subseteq I\times D\to I\times A$. In other words, 
\[UV=\{(i,a)\in I\times A\mid \exists b,c: A\s (U(i,b)\wedge V(i,c)\wedge bc=a)\}.\]
A filter $\psi$ on $I^\ast(A)$ is then a set of fiberwise inhabited subobjects of $I\times A$ that is upwards closed and closed under the application above.
Since $I^\ast$ is a regular functor, we know that such a filter $\psi$ is combinatorially complete whenever it extends $I^\ast(\phi)$.

Suppose that a PCA $(I^\ast(A),\psi)$ over $\mathcal{C}/I$ is given. An assembly $X$ over this PCA is an object $|X|\stackrel{k_X}{\longrightarrow} I$ of $\mathcal{C}/I$, together with a subobject $E_X\subseteq |X|\times A$ that is total in $\mathcal{C}$, i.e., such that $\models_{x:|X|}\s \exists a:A\s (E_X(x,a))$.
Now consider another assembly $Y$, and suppose that $f\colon |X|\to |Y|$ is an arrow of $\mathcal{C}/I$, that is, $k_Y\circ f = k_X$. Then a tracker of $f\colon X\to Y$, spelled out in terms of the internal logic of $\mathcal{C}$, is an inhabited $U\subseteq I\times A$ such that
\[k_X(x)=i\wedge E_X(x,a)\wedge U(i,r)\models_{x:|X|;i:I;r,a:A} ra\denotes\wedge\ E_Y(f(x),ra).\]
That is, if $(i,r)$ is in $U$, then $r$ must track $f$ as if it were a morphism $(|X|,E_X)\to(|Y|,E_Y)$ in $\asm(A,\phi)$, but only for those $x\in |X|$ that lie in the fiber of $i$.

Now let $I$ be an assembly over $(A,\phi)$. By definition, $E_I$ is an inhabited subobject of $|I|^\ast(A)$ in $\mathcal{C}/|I|$. We will show that $\asm(A,\phi)/I$ is equivalent to $\asm(|I|^\ast(A),\phi_I)$, where 
\[\phi_I = \langle |I|^\ast(\phi)\cup\{E_I\}\rangle.\]
Observe that $\phi_I$ is combinatorially complete since it extends $|I|^\ast(\phi)$. First, we give an alternative characterization of $\phi_I$.
\begin{lem}\label{phii_expl_lem}
Let $I$ be an assembly over the PCA $(A,\phi)$. Then
\begin{align}\label{phii_explicitly}
\phi_I = \{U\subseteq |I|\times A\mid \exists V\in\phi\s (|I|^\ast(V)\cdot E_I\denotes\wedge\ |I|^\ast(V)\cdot E_I\subseteq U)\}.
\end{align}
\end{lem}
\begin{proof}
First of all, suppose that $U\in\phi$. Then also $\mathsf{K}U\in\phi$. Moreover, we know that $|I|^\ast(\mathsf{K})$ realizes $\lambda xy.x$, so we see that $|I|^\ast(\mathsf{K}U)\cdot E_I\denotes$ and
\[|I|^\ast(\mathsf{K}U)\cdot E_I = |I|^\ast(K)\cdot |I|^\ast(U)\cdot E_I = |I|^\ast(U).\]
This shows that the right-hand side of \eqref{phii_explicitly} contains $|I|^\ast(\phi)$. Furthermore, since $|I|^\ast(\mathsf{I})$ realizes $\lambda x.x$, we have that $|I|^\ast(\mathsf{I})\cdot E_I = E_I$, so $E_I$ belongs to the right-hand side of \eqref{phii_explicitly} as well. This means that the right-hand side of \eqref{phii_explicitly} extends $|I|^\ast(\phi)\cup\{E_I\}$.

Moreover, any filter extending $|I|^\ast(\phi)\cup\{E_I\}$ must clearly extend the right-hand side of \eqref{phii_explicitly}, so it remains to show that this is actually a filter; upwards closure is obvious.

Suppose that we have $U,U'\subseteq |I|\times A$ for which there exist $V,V'\in\phi$ with $|I|^\ast(V)\cdot E_I\subseteq U$ and $|I|^\ast(V')\cdot E_I\subseteq U'$, and such that $UU'\denotes$. Then $\mathsf{S}VV'\in\phi$. Moreover, we know that $|I|^\ast(\mathsf{S})$ realizes $\lambda xyz.xz(yz)$, so we see that $|I|^\ast(\mathsf{S}VV')\cdot E_I\denotes$ and
\[|I|^\ast(\mathsf{S}VV')\cdot E_I = |I|^\ast(\mathsf{S})\cdot |I|^\ast(V)\cdot |I|^\ast (V')\cdot E_I \subseteq (|I|^\ast(V)\cdot E_I)\cdot (|I|^\ast (V')\cdot E_I) \subseteq UU',\]
as desired.
\end{proof}
\begin{thm}\label{slice}
Let $I$ be an assembly over the PCA $(A,\phi)$. Then the categories $\asm(A,\phi)/I$ and $\asm(|I|^\ast(A),\phi_I)$ are equivalent.
\end{thm}
\begin{proof}
As usual, let $\mathcal{C}$ be the underlying regular category of $(A,\phi)$. We define the required pseudo inverses
\[\begin{tikzcd}
{\asm(A,\phi)/I} \arrow[rr, shift left, "F"] &  & {\asm(|I|^\ast(A),\phi_I)} \arrow[ll, shift left, "G"].
\end{tikzcd}\]
An object of $\asm(A,\phi)/I$ is a morphism of assemblies $l_X\colon X\to I$. We define the assembly $FX$ over $(|I|^\ast(A),\phi_I)$ simply by $|FX| = |X|$, $k_{FX} = l_X$ and $E_{FX} = E_X$. Moreover, given a commutative triangle
\[\begin{tikzcd}
X \arrow[rr, "f"] \arrow[rd, "l_X"'] &   & Y \arrow[ld, "l_Y"] \\
                                     & I &                    
\end{tikzcd}\]
in $\asm(A,\phi)$, we define $Ff\colon FX\to FY$ simply as $f$. If $U\in\phi$ tracks $f\colon X\to Y$, then $|I|^\ast(U)\in\phi_I$ tracks $Ff\colon FX\to FY$, so we see that $F$ is well-defined, and clearly, $F$ is a functor.

Conversely, suppose that $X$ is an object of $\asm(|I|^\ast(A),\phi_I)$. We define the object $l_{GX}\colon GX\to I$ of $\asm(A,\phi)/I$ by $|GX| = |X|$,
\[E_{GX} = \{(x,a)\in |X|\times A\mid \exists b,c,d: A\s (\mathsf{P}(b)\wedge E_I(k_X(x),c)\wedge E_X(x,d)\wedge bcd = a)\},\]
and $l_{GX} = k_X$. This is clearly an assembly, and $k_X$ is a morphism $GX\to I$ in $\asm(A,\phi)$, since it is tracked by $\mathsf{P}_0$. For an arrow $f\colon X\to Y$ in $\asm(|I|^\ast(A),\phi_I)$, we define $Gf\colon GX\to GY$ simply as $f$. By definition, we have $l_{GY}\circ Gf = k_Y\circ f = k_X = l_{GX}$, but we need to verify that $Gf$ is a morphism $GX\to GY$ in $\asm(A,\phi)$. Let $U\in\phi_I$ be a tracker of $f\colon X\to Y$. By \eqref{phii_explicitly}, there exists a $V\in\phi$ such that $|I|^\ast(V)\cdot E_I\subseteq U$. One now easily verifies that every realizer of
\[\lambda x. \mathsf{P}(\mathsf{P}_0 x)(V(\mathsf{P}_0 x)(\mathsf{P}_1 x)),\]
tracks $Gf$, which shows that $Gf$ is indeed a morphism $GX\to GY$. We conclude that $G$ is a well-defined functor.

It remains to show that $F$ and $G$ are pseudo inverses. If $l_X\colon X\to I$ is in $\asm(A,\phi)/I$, then applying $F$ and $G$ yields the object $l_{GFX}\colon GFX\to I$, where $|GFX| = |X|$ and $l_{GFX} = l_X$, but
\[E_{GFX} = \{(x,a)\in |X|\times A\mid \exists b,c,d: A\s (\mathsf{P}(b)\wedge E_I(l_X(x),c)\wedge E_X(x,d)\wedge bcd = a)\}.\]
In order to show that $GF\simeq\id$, it therefore suffices to show that the identity on $|X|$ is both a morphism $X\to GFX$ and a morphism $GFX\to X$. If $U\in\phi$ tracks $l_X$, then every realizer of $\lambda x.\mathsf{P}(Ux)x$ tracks $\id_{|X|}\colon X\to GFX$, and in the other direction, $\mathsf{P}_1$ tracks $\id_{|X|}\colon GFX\to X$.

If $X$ is an object of $\asm(|I|^\ast(A),\phi_I)$, then $|FGX| = |X|$ and $k_{FGX} = k_X$, but
\[E_{FGX} = \{(x,a)\in |X|\times A\mid \exists b,c,d:A\s (\mathsf{P}(b)\wedge E_I(k_X(x),c)\wedge E_X(x,d)\wedge bcd = a)\}.\]
As above, it suffices to show that the identity on $|X|$ is a morphism $X\to FGX$ and $FGX\to X$. This is indeed the case, since $|I|^\ast(\mathsf{P})\cdot E_I\in\phi_I$ tracks $\id_{|X|}\colon X\to FGX$, whereas $|I|^\ast(\mathsf{P}_1)\in\phi_I$ tracks $\id_{|X|}\colon FGX\to X$. This completes the proof.
\end{proof}

\begin{defi}
If $(A,\phi)$ is a PCA and $I$ is an assembly, then we denote the PCA $(|I|^\ast(A),\phi_I)$ by $(A,\phi)/I$. \ruitje
\end{defi}
In this way, we may reformulate \cref{slice} above as $\asm(A,\phi)/I\simeq \asm((A,\phi)/I)$.

Before we treat our first example, we need the following result.
\begin{prop}
Under the equivalence $\asm((A,\phi)/I)\simeq \asm(A,\phi)/I$ from \cref{slice}, the constant objects of $\asm((A,\phi)/I)$ correspond to objects of $\asm(A,\phi)/I$ that are prone morphisms in $\asm(A,\phi)$.
\end{prop}
\begin{proof}
This is a straightforward calculation using the explicit description of the equivalence $\asm((A,\phi)/I)\simeq \asm(A,\phi)/I$ above.
\end{proof}

\begin{ex}\label{pullback_asm}
Let $(A,\phi)$ be a PCA over $\mathcal{C}$. If $f\colon I\to J$ is a morphism of assemblies, then we also have a pullback functor $f^\ast\colon\mathcal{C}/|J|\to \mathcal{C}/|I|$, which is a regular functor. Moreover, we have the pullback functor $f^\ast\colon \asm(A,\phi)/J\to\asm(A,\phi)/I$, which is regular. By \cref{slice}, we can also view $f^\ast$ as a functor $\asm((A,\phi)/J)\to\asm((A,\phi)/I)$, which clearly commutes with $\Gamma$. Moreover, we know that $f^\ast\colon \asm(A,\phi)/J\to\asm(A,\phi)/I$ preserves prone morphisms, so $f^\ast\colon\asm((A,\phi)/J)\to\asm((A,\phi)/I)$ preserves constant objects, i.e., commutes with $\nabla$. By \cref{image_asm}, $f^\ast$ must be induced by an applicative morphism $(A,\phi)/J\to(A,\phi)/I$.

Let us describe this applicative morphism explicitly. The fact that $f$ is a morphism means precisely that $f^\ast(E_J)\subseteq f^\ast(|J|^\ast(A)) = |I|^\ast(A)$ belongs to $\phi_I$. This yields $f^\ast(|J|^\ast(\phi)\cup\{E_J\})\subseteq \phi_I$, and hence also
\[\langle f^\ast(\phi_J)\rangle = \langle f^\ast(|J|^\ast(\phi)\cup\{E_J\})\rangle\subseteq\phi_I.\]
Therefore, $(f^\ast, \delta_{|I|^\ast(A)})$ is an applicative morphism $(A,\phi)/J\to(A,\phi)/I$. A direct calculcation shows that, under the equivalence from \cref{slice}, $\asm(f^\ast, \delta_{|I|^\ast(A)})$ is indeed naturally isomorphic to the pullback functor $f^\ast\colon \asm(A,\phi)/J\to\asm(A,\phi)/I$. 

If we let $J$ be the terminal assembly, then we see that $(|I|^\ast,\delta_{|I|^\ast(A)})$ is an applicative morphism $(A,\phi)\to(A,\phi)/I$. Moreover, under the equivalence from \cref{slice}, $\asm(|I|^\ast,\delta_{|I|^\ast(A)})$ is naturally isomorphic to $I^\ast\colon\asm(A,\phi)\to\asm(A,\phi)/I$.
\ruitje
\end{ex}

\begin{ex}
Let $(A,C)$ be a classical relative PCA, and let $I$ be an assembly over $(A,C)$. We know that the category $\mathsf{Set}/|I|$ is equivalent to $\mathsf{Set}^{|I|}$. This means that $(A,C)/I$ is isomorphic to the PCA $((A)_{i\in |I|}, C_I)$ over $\mathsf{Set}^{|I|}$, where
\[C_I = \{(U_i\subseteq A)_{i\in |I|}\mid \exists r\in C\s \forall i\in |I|\s \forall a\in A\s (E_I(i,a)\to ra\denotes\wedge\ ra\in U_i)\}.\]
We may picture this as follows. Each coordinate $i\in |I|$ is labelled with information: the set of $a\in A$ such that $E_I(i,a)$. A sequence $(U_i)_{i\in |I|}$ of subsets of $A$ is in the filter $C_I$ precisely when there exists an algorithm that, uniformly in $i\in |I|$, turns the information we have about $i$ into elements of $U_i$. \ruitje
\end{ex}

\begin{ex}\label{slice_partitioned}
Let $(A,C)$ be a PCA over $\mathcal{C}$ that is generated by singletons and consider a \emph{partitioned assembly} $I$. That is, we require $E_I(i,a)\wedge E_I(i,a')\models_{i:I;a,a':A} a=a'$. This means that there exists an arrow $f\colon |I|\to A$ such that $E_I = \{(i,f(i))\mid i\in |I|\}$. Then $(A,C)/I=(I|^\ast(A),C_I)$, where
\begin{align*}
C_I &= \{U\subseteq |I|\times A\mid \exists V\in\phi_C\s (|I|^\ast(V)\cdot E_I\denotes\wedge\ |I|^\ast(V)\cdot E_I\subseteq U)\}\\
&=\{U\subseteq |I|\times A\mid \exists r\in C\s (|I|^\ast(r)\cdot E_I\denotes\wedge\ |I|^\ast(r)\cdot E_I\subseteq U)\}\\
&=\{U\subseteq |I|\times A\mid \exists r\in C\s (\models_{i:I} r\cdot f(i)\denotes\wedge\ U(i,r\cdot f(i)))\}.
\end{align*}
This filter is generated by the set of global sections
\[\{g\colon |I|\to A\mid \exists r\in C\s (\models_{i:I} r\cdot f(i)\denotes\wedge\ r\cdot f(i)=g(i))\}.\]
of $|I|^\ast(A)$. In other words, a global section of $|I|^\ast(A)$, which is a function $|I|\to A$, counts as computable iff it can be computed uniformly in terms of the `basic' element $f$.

We see that being generated by singletons is preserved under slicing over partitioned assemblies. Below (\cref{resub_classifier}), we shall see that it is not preserved by slicing over general assemblies. \ruitje
\end{ex}

We consider a few interesting examples of slicing over partitioned assemblies.
\begin{ex}
Consider a classical relative PCA $(A,C)$. Let $I$ be the partitioned assembly defined by $|I| = 2= \{0,1\}$, $E_I(0,\mathsf{k})$ and $E_I(1,\bar{\mathsf{k}})$, where $\mathsf{k}$ and $\bar{\mathsf{k}}$ are the usual combinators. Then $I$ is the coproduct $1+1$.

Then $(A,C)/(1+1)$ is isomorphic to the PCA $((A,A),C_{1+1})$ over $\mathsf{Set}^2$. The filter $C_{1+1}$ is generated by $C\times C\subseteq A\times A \cong \mathsf{Set}^2(1,(A,A))$. In other words, it contains all pairs $(U_0,U_1)$ such that $U_0\cap C\neq\emptyset$ and $U_1\cap C\neq\emptyset$. From this, it easily follows that $\asm((A,A),C_{1+1})$ is isomorphic to $\asm(A,C)^2$. This is, of course, no surprise: since $\asm(A,C)$ is a quasitopos, we already knew that the categories $\asm(A,C)/(1+1)$ and $\asm(A,C)^2$ should be equivalent. \ruitje
\end{ex}

\begin{ex}
Let $(A,C)$ be a PCA over $\mathcal{C}$ that is generated by singletons and let $r_0\in C$ be a fixed element. (Observe that $C$ is nonempty since $\phi_C$ should be combinatorially complete.) For an object $X$ of $\mathcal{C}$, we consider the partitioned assembly $I$ determined by $|I| = X$ and the function $X\to A$ assuming the constant value $r_0$. Then $I$ is the constant object $\nabla X$. Using \cref{slice_partitioned}, we see that $C_{\nabla X}$ is generated by all $g\colon X\to A$ that factor through some element from $C$. In other words, $U\subseteq X\times A$ belongs to $C_{\nabla X}$ if and only if there is a computable element $r\in C$ such that $X\times r\subseteq U$.

If $\mathcal{C}=\mathsf{Set}$, then we may take $X=2$. Then 
\[C_{\nabla 2} = \{(U_0,U_1)\subseteq(A,A)\mid U_0\cap U_1\cap C\neq\emptyset\},\]
which is generated by the `diagonal' $\{\{(c,c)\}\subseteq (A,A)\mid c\in C\}$. We see that $\asm(A,C)/\nabla 2$ is $\asm(A,C)^2$, except that all realizing in the two coordinates needs to take place \emph{simultaneously}.

If we take $C=A$, so that we start with a classical absolute PCA, then $C_{\nabla 2}$ is \emph{not} the maximal filter on $(A,A)$ (barring the case where $A$ is the zero PCA 1). This means that $((A,A),C_{\nabla 2})$ is not (equivalent to) an absolute PCA. So even if we start out with an absolute PCA, we may end up with a non-absolute PCA after slicing.

Therefore, in contrast with taking products (see Remark \cref{remarks_product}(ii)), slicing forces us to consider non-absolute notions of computability. \ruitje
\end{ex}

\begin{ex}
Consider Kleene's first model $\mathcal{K}_1$ (equipped with the maximal filter $\phi=\mathcal{P}^\ast\mathbb{N}$). We define the partitioned assembly $N$ by $|N|=\mathbb{N}$ and $E_N=\delta_\mathbb{N}$; this is the natural numbers object in $\asm(\mathcal{K}_1)$. Then $\mathcal{K}_1/N$ is isomorphic to the PCA $((\mathcal{K}_1)_{n\in\mathbb{N}},\phi_N)$ over $\mathsf{Set}^\mathbb{N}$, where the filter $\phi_N$ is generated by all \emph{recursive} infinite sequences of natural numbers (more precisely, all global sections $(\{a_n\})_{n\in\mathbb{N}}$, where $(a_n)_{n\in\mathbb{N}}$ is recursive). Therefore $\asm(\mathcal{K}_1)/N$ is equivalent to a category whose objects are infinite sequences of assemblies, and whose arrows are infinite sequences of morphisms of assemblies for which there exists a recursive sequence of trackers.

We know that the product PCA $\mathcal{K}_1^\mathbb{N}$, on the other hand, can be obtained by taking the maximal filter on $(\mathcal{K}_1)_{n\in\mathbb{N}}$, which is generated by \emph{all} infinite sequences of natural numbers. In fact, one can show that $\asm(\mathcal{K}_1)/N$ and $\asm(\mathcal{K}_1)^\mathbb{N}$ are \emph{not} equivalent. \ruitje
\end{ex}


\begin{ex}
Consider Scott's graph model $\mathcal{P}\omega$, equipped with the maximal filter $\phi$. The natural numbers object $N$ of $\asm(\mathcal{P}\omega)$ may be defined by $|N|=\mathbb{N}$ and $E_N(n) = \{\{n\}\}$ for $n\in\mathbb{N}$. For every infinite sequence $(V_n)_{n\in N}$ of elements of $\mathcal{P}\omega$, there exists a Scott continuous function $F\colon\mathcal{P}\omega\to\mathcal{P}\omega$ such that $F(\{n\}) = V_n$ for all $n\in\mathbb{N}$. This implies that $\phi_N$ is the maximal filter on $(\mathcal{P}\omega)_{n\in\mathbb{N}}$, so it follows that $\mathcal{P}\omega/N$ and $\mathcal{P}\omega^\mathbb{N}$ are isomorphic. In particular, $\asm(\mathcal{P}\omega)/N$ and $\asm(\mathcal{P}\omega)^\mathbb{N}$ are equivalent. \ruitje
\end{ex}

Now we consider an example of slicing over a non-partitioned assembly.

\begin{ex}\label{resub_classifier}
Again, consider $\mathcal{K}_1$ with the maximal filter $\phi$. Define the assembly $I$ by $|I| = \{0,1\}$ and
\[E_I = \{(x,n)\in 2\times\mathbb{N}\mid (x=0\wedge n\in K)\vee(x=1\wedge n\not\in K)\},\]
where $K$ is the standard set $\{n\in\mathbb{N}\mid \varphi_n(n)\!\denotes\}$. This assembly is also known as the \emph{r.e.\@ subobject classifier}, which is usually denoted by $\Sigma$. It is easy to see that this assembly cannot be isomorphic to a partitioned assembly. We will show that $\phi_\Sigma$ is not generated by singletons. Suppose that the set of singletons $C$ generates $\phi_\Sigma$. By \cref{phii_expl_lem}, $\phi_\Sigma$ consists of all pairs $(U,V)$ for which there exists a (total) recursive function $f$ satisfying: if $n\in K$, then $f(n)\in U$, whereas if $n\not\in K$, then $f(n)\in V$. This means that the pair $(K,\mathbb{N}\backslash K)$ is certainly in $\phi_\Sigma$, so there must exist an $(\{a\},\{b\})\in C$ with $a\in K$ and $b\not\in K$. Since $(\{a\},\{b\})$ must be in $\phi_\Sigma$, we see that there exists a recursive function $f\colon\mathbb{N}\to\mathbb{N}$ such that $f(n) = a$ for all $n\in K$ and $f(n)=b$ for all $n\not\in K$. Since $a\neq b$, this implies that $K$ is decidable, which we know to be false. (In fact, this argument is very similar to the argument needed to show that $\Sigma$ cannot be partitioned.) \ruitje
\end{ex}

\begin{ex}
Using \cref{resub_classifier} above, we may also construct a PCA over $\mathsf{Set}$ that is not generated by singletons, and is therefore not a PCA in the classical sense. Let $\Pi\colon\mathsf{Set}^2\to\mathsf{Set}$ be the product functor sending a pair $(A,B)$ of sets to $A\times B$. This functor is regular, so we can consider the PCA $\Pi^\ast(\mathcal{K}_1/\Sigma)$. Its underlying set is $\mathbb{N}\times\mathbb{N}$ with the coordinatewise application, and its filter is $\langle\Pi(\phi_\Sigma)\rangle$, which is easily shown to be equal to 
\[\uparrow\!(\Pi(\phi_\Sigma)) = \{W\subseteq\mathbb{N}\times\mathbb{N}\mid\exists (U,V)\in \phi_\Sigma\s (U\times V\subseteq W)\}.\]
Let us denote this filter by $\psi$, and suppose that it is generated by the set of singletons $C$. We certainly have that $K\times(\mathbb{N}\backslash K)\in\psi$, so there must exist a singleton $\{(a,b)\}\in C$ such that $a\in K$ and $b\not\in K$. But now we have $(\{a\},\{b\})\in\phi_\Sigma$ and $a\neq b$, which is impossible as we have seen above. \ruitje
\end{ex}

\section{Computational Density}\label{section_cd}

Let $p\colon \mathcal{D}\to\mathcal{C}$ be a geometric morphism between \emph{toposes}. Then the inverse image $p^\ast\colon \mathcal{C}\to\mathcal{D}$ is always a regular functor, which means we can apply the theory developed above. Of course, the internal logic of toposes is quite a bit stronger than that of regular categories, since toposes interpret full intuitionistic (typed) higher-order logic. For the larger part of this section, we will only need first-order logic. As for regular formulas, if $\varphi(x_1, \ldots, x_n)$ is a first-order formula and $x_i:X_i$, then we write
\[\{(x_1, \ldots, x_n)\in X_1\times \cdots \times X_n\mid \varphi(\vec{x})\}\]
for its interpretation. If $\varphi$ is a first-order sentence, then we write $\mathcal{C}\models \varphi$ to indicate that this sentence is valid in the topos $\mathcal{C}$.

\begin{rem}
The presence of full first-order logic also allows for a somewhat more elegant treatment of the material in the preceding sections. For example, if $A$ is a PAS over $\mathcal{C}$ and $t(x_0, \ldots, x_n)$ is a term, then we may the define the object of all elements that realize it. Explicitly, we can define $\lval\lambda\vec{x}.t\rval\subseteq A$ as:
\[\{r\in A\mid \forall \vec{a}: A^n\s (r\vec{a}\denotes\wedge\ \forall b: A\s (t(\vec{a},b)\denotes\to r\vec{a}b\denotes\wedge\ r\vec{a}b=t(\vec{a},b)))\}.\]
In other words, this is the largest possible subobject of $A$ that can realize $\lambda\vec{x}.t$. If $\phi$ is a filter, then by upwards closure, it contains a realizer of $\lambda\vec{x}.t$ if and only if $\lval\lambda\vec{x}.t\rval\in\phi$. Moreover, the fact that a regular functor $p$ preserves realizers can now be expressed as $p(\lval\lambda\vec{x}.t\rval_A)\subseteq \lval\lambda\vec{x}.t\rval_{p(A)}$ (where the subscript indicates w.r.t.\@ which PAS we compute the suobject of realizers). Similar notation can be introduced for trackers of morphisms of assemblies, trackers of applicative morphisms, and realizers of inequalities between relations. This is, in fact, the approach taken in \cite{wps}, where all the base categories $\mathcal{C}$ are Heyting categories and therefore soundly interpret intuitionistic first-order logic.

Even though toposes themselves are able to interpret full first-order logic, the morphisms betweem them that we consider typically do not preserve all this structure. This is why we have chosen to work with regular categories in the preceding sections. By formulating the key concepts in terms of regular sequents, we were able to see that regular functors suffice to transport the necessary structure from one base category to another. \ruitje
\end{rem}

Suppose that $(p^\ast,f)$ is an applicative morphism $(A,\phi)\to(B,\psi)$, and consider the induced functor $\asm(p^\ast,f)\colon \asm(A,\phi)\to\asm(B,\psi)$. The goal of this section is to answer the following question: under which conditions does $\asm(p^\ast,f)$ have a right adjoint? In the case of classical PCAs, the answer is given by a notion called \emph{computational density} \cite{hofstrajaap}, which we define now.
\begin{defi}\label{classical_cd}
Let $(A,C)$ and $(B,D)$ be classical relative PCAs and let $f\colon (A,C)\to (B,D)$ be an applicative morphism. We say that $f$ is \emph{computationally dense} if there exists an $\mathsf{m}\in D$ such that for all $s\in D$, there exists an $r\in C$ with:
\[\mbox{for all }a\in A,\mbox{ if }s\cdot f(a)\denotes,\mbox{ then }ra\denotes, \mathsf{m}\cdot f(ra)\denotes\mbox{ and }\mathsf{m}\cdot f(ra)\subseteq s\cdot f(a).\]
Here we have written $f(a)$ for the set of $b$ such that $f(a,b)$, and an expression such as $d\cdot f(a)$ denotes $\{db\mid b\in f(a)\}$. \ruitje
\end{defi}
The property defining computational density is of too high logical complexity to make sense in an arbitrary regular category. Therefore, we will only generalize this notion to toposes. However, Johnstone has reformulated the notion of computational density into a much simpler property (\cite{johnstone}, Lemma 3.2), which \emph{does} work for regular categories.
\begin{defi}
Let $(A,C)$ and $(B,D)$ be classical relative PCAs and let $f\colon (A,C)\to (B,D)$ be an applicative morphism. We say that $f$ is \emph{quasi-surjective} if there exists an $\mathsf{m}\in D$ such that for all $s\in D$, there exists an $r\in C$ with $\mathsf{n}\cdot f(r) = s$. (I.e., $\mathsf{n}b$ is defined and equal to $s$, for every $b\in f(r)$.) \ruitje
\end{defi}
Johnstone's proof in fact shows that, in \cref{classical_cd}, it makes no difference if we require $ra$ to be always defined (although one may have to adjust $\mathsf{m}$ in order to achieve this). We now give the appropriate generalizations of the two notions above to our setting.

\begin{defi}
Let $(p,f)\colon (A,\phi)\to (B,\psi)$ be an applicative morphism.
\begin{itemize}
\item[(i)]	$(p,f)$ is called \emph{quasi-surjective} if there exists an $\mathsf{N}\in\psi$ such that for all $U\in\psi$, there exists a $V\in\phi$ satisfying $\mathsf{N}\cdot f(p(V))\denotes$ and $\mathsf{N}\cdot f(p(V))\subseteq U$.
\item[(ii)]	Now suppose that the underlying category $\mathcal{D}$ of $(B,\psi)$ is a topos. Then $(p,f)$ is called \emph{computationally dense} if there exists an $\mathsf{M}\in\psi$ such that for all $U\in\psi$, there exists a $V\in\phi$ such that $V\!A\denotes$ and
\[\mathcal{D}\models \forall r,a: p(A)\s(p(V)(r)\wedge U\cdot f(a)\denotes\ \to \mathsf{M}\cdot f(ra)\subseteq U\cdot f(a)),
\]
where $U\cdot f(a)\denotes$ abbreviates the formula $\forall u,b: B\s (U(u)\wedge f(a,b)\to ub\denotes)$, and $\mathsf{M}\cdot f(ra)\subseteq U\cdot f(a)$ abbreviates
\[\forall m,b: B\s (\mathsf{M}(m)\wedge f(ra,b)\to (mb\denotes\wedge\ \exists u,b': B\s (U(u)\wedge f(a,b')\wedge ub'\denotes \wedge\ ub'=mb))).\]
\end{itemize}
We say that such $\mathsf{N}$ and $\mathsf{M}$ \emph{witness} the quasi-surjectivity resp.\@ the computational density of $(p,f)$. \ruitje 
\end{defi}


\begin{rem}
Observe that our requirement that $V\!A\denotes$ implies that $p(V)\cdot p(A)\denotes$ as well, so that $\mathcal{D}\models\forall r,a:p(A)\s (p(V)(r)\to ra\denotes)$. As we shall see later, it is not sufficient, for our purposes, to require merely that $p(V)\cdot p(A)\denotes$. \ruitje
\end{rem}
\begin{rem}
When reasoning internally in a topos, we will also use expressions such as $U\cdot f(a)\denotes$, trusting that the reader can formulate those as proper first-order statements if desired. Alternatively, the reader can think of $f$ as a map from $A$ into the power object of $B$, rather than a relation between $A$ and $B$. However, most of what follows only requires first-order internal reasoning; we will not need the presence of power objects until the proof of \cref{cd_is_rightadjoint}(ii). \ruitje
\end{rem}

First of all, we study the quasi-surjective applicative morphisms.

\begin{ex}\label{reindexing=qs}
Let $(A,\phi)$ be a PCA over the regular category $\mathcal{C}$ and let $p\colon\mathcal{C}\to\mathcal{D}$ be a regular functor. Then $(p,\delta_{p(A)})\colon (A,\phi)\to p^\ast(A,\phi)$ is quasi-surjective. Indeed, as in the proof of \cref{phii_expl_lem}, one can show that
\[\langle p(\phi)\rangle = \{U\subseteq p(A)\mid \exists V\in\phi\s (p(V)\cdot p(A)\denotes\wedge\ p(V)\cdot p(A)\subseteq U)\}.\]
Since $p(A)\in \langle p(\phi)\rangle$, there exists a realizer $\mathsf{N}\in\langle p(\phi)\rangle$ of $\lambda x. x\cdot p(A)$, and this $\mathsf{N}$ witnesses the quasi-surjectivity of $(p,\delta_{p(A)})$. \ruitje
\end{ex}

\begin{ex}\label{pullback=qs}
Let $(A,\phi)$ be a PCA over $\mathcal{C}$ and let $I$ be an assembly. Then the applicative morphism $(|I|^\ast,\delta_{|I|^\ast(A)})\colon (A,\phi)\to(A,\phi)/I$ from \cref{pullback_asm} is quasi-surjective. Indeed, by \cref{phii_expl_lem}, we have that any realizer $\mathsf{N}\in\phi_I$ of $\lambda x. x\cdot E_I$ witnesses the fact that $(|I|^\ast,\delta_{|I|^\ast(A)})$ is quasi-surjective. A fortiori, all the applicative morphisms $(A,\phi)/J\to(A,\phi)/I$ discussed in \cref{pullback_asm} are quasi-surjective. \ruitje
\end{ex}

\begin{prop}\label{PCAqs}
\begin{itemize}
\item[\textup{(}i\textup{)}]	PCAs, quasi-surjective applicative morphisms and applicative transformations form a 2-category $\mathsf{PCA}_{\textup{qs}}$.
\item[\textup{(}ii\textup{)}]	If $(A,\phi)\stackrel{(p,f)}{\longrightarrow} (B,\psi)\stackrel{(q,g)}{\longrightarrow} (C,\chi)$ are applicative morphisms such that $(q,g)\circ(p,f)$ is quasi-surjective, then $(q,g)$ is quasi-surjective as well.
\end{itemize}
\end{prop}
\begin{proof}
For (i), we only need to show that quasi-surjective applicative morphisms are closed under identities and composition. If $(A,\phi)$ is a PCA, then $\mathsf{I}$ witnesses the quasi-surjectivity of $\id_{(A,\phi)}$. Now suppose that $(A,\phi)\stackrel{(p,f)}{\longrightarrow} (B,\psi)\stackrel{(q,g)}{\longrightarrow} (C,\chi)$ are quasi-surjective applicative morphisms, and let $\mathsf{N}\in\psi$ and $\mathsf{N}'\in\chi$ witness the quasi-surjectivity of $(p,f)$ and $(q,g)$, respectively. Furthermore, let $T\in\chi$ be a tracker of $g$. We claim that any realizer $\mathsf{N}''\in\chi$ of
\[\lambda x.\mathsf{N}'(T\cdot g(q(\mathsf{N}))\cdot x)\]
witnesses the quasisurjectivity of $(q,g)\circ (p,f) = (qp, g\circ q(f))$. Suppose that a $U\in\chi$ is given. Let $V\in\psi$ such that $\mathsf{N}'\cdot g(q(V))\subseteq U$ and let $W\in\phi$ such that $\mathsf{N}\cdot f(p(W))\subseteq V$. Then we also have $q(\mathsf{N})\cdot q(f(p(W)))\subseteq q(V)$ and therefore
\[T\cdot g(q(\mathsf{N}))\cdot g(q(f(p(W))))\subseteq g(q(V)).\]
Since $q$ is a regular functor, we have that \[g(q(f(p(W)))) = g(q(f)(q(p(W)))) = (g\circ q(f))((qp)(W)),\]
so it follows from the above that $\mathsf{N}''\cdot(g\circ q(f))((qp)(W))$ is defined and a subobject of $U$, as desired.

For (ii), let $\mathsf{N}\in\chi$ witness the quasi-surjectivity of $(qp,g\circ q(f))$, and let $U\in\chi$. Then there exists a $V\in\phi$ such that 
\[\mathsf{N}\cdot g(q(f(p(V)))) = \mathsf{N}\cdot (g\circ q(f))((qp)(V))\subseteq U.\]
We know that $f(p(V))\in\psi$, so we conclude that $\mathsf{N}$ also witnesses the quasi-surjectivity of $(q,g)$.
\end{proof}

We constructed the 1-category $\mathsf{PCA}$ in such a way that the forgetful functor $\mathsf{PCA}\to\mathsf{REG}$ is an opfibration. \cref{reindexing=qs} and \cref{PCAqs} yield that the same is true for $\mathsf{PCA}_\text{qs}\to\mathsf{REG}$.

We now proceed to show that, when both these notions apply, computational density and quasi-surjectivity coincide. The proof is an easy adaptation of Johnstone's argument from \cite{johnstone}.

\begin{prop}
Let $(p,f)\colon (A,\phi)\to (B,\psi)$ be an applicative morphism, and suppose that the underlying category $\mathcal{D}$ of $(B,\psi)$ is a topos. Then $(p,f)$ is computationally dense if and only if it is quasi-surjective.
\end{prop}
\begin{proof}
First, suppose that $(p,f)$ is computationally dense, witnessed by $\mathsf{M}\in\psi$. Let $T$ track $f$ and let $\mathsf{N}\in\psi$ realize
\[\lambda x.\mathsf{M}(T\cdot x\cdot f(p(A))).\]
Suppose that $U\in\psi$. Then $\mathsf{K}U\in\psi$ as well. Let $V\in\phi$ such that $V\!A\denotes$ and 
\[\mathcal{D}\models \forall r,a:p(A)\s(p(V)(r)\wedge \mathsf{K}U\cdot f(a)\denotes\ \to \mathsf{M}\cdot f(ra)\subseteq \mathsf{K}U\cdot f(a)).\]
Now reason inside $\mathcal{D}$ and let $a\in p(A)$ be arbitrary. Then $\mathsf{K}U\cdot f(a)\denotes$ always holds. It follows that for all $r\in p(A)$ and $m,b\in B$: if $p(V)(r)$, $\mathsf{M}(m)$ and $f(ra,b)$, then $mb\denotes$ and $mb\in U$. From this, we can conclude (externally) that $\mathsf{M}\cdot f(p(V)\cdot p(A))\denotes$ and $\mathsf{M}\cdot f(p(V)\cdot p(A))\subseteq U$. This yields
\[\mathsf{N}\cdot f(p(V)) \subseteq \mathsf{M}(T \cdot f(p(V))\cdot f(p(A))) \subseteq \mathsf{M}\cdot f(p(V)\cdot p(A)) \subseteq U,\]
as desired.

For the converse, suppose that $(p,f)$ is quasi-surjective, witnessed by $\mathsf{N}\in\psi$. Again, let $T\in\psi$ track $f$, and consider a realizer $\mathsf{M}\in\psi$ of
\[\lambda x.\mathsf{N}(T\cdot f(\mathsf{P}_0)\cdot x)(T\cdot f(\mathsf{P}_1)\cdot x).\]
Let $U\in\psi$, and find a $V\in\phi$ such that $\mathsf{N}\cdot f(p(V)) \subseteq U$. We show that $W:=\mathsf{P} V$ has the required properties. First of all, we observe that $\mathsf{P}V\!A$ is defined. Moreover, $p(W) = p(\mathsf{P})\cdot p(V)$, and as we remarked earlier, $p(\mathsf{P})$ is a pairing combinator w.r.t.\@ $p(A)$. Now we reason internally in $\mathcal{D}$, let $r,a\in p(A)$ be arbitrary and suppose that $p(V)(r)$ and $U\cdot f(a)\denotes$.  Since $ra\in p(W)\cdot a = p(\mathsf{P})\cdot p(V)\cdot a$, we see that
\[T\cdot f(\mathsf{P}_0)\cdot f(ra) \subseteq f(\mathsf{P}_0\cdot (ra))\subseteq f(p(V))\quad\mbox{and}\quad T\cdot f(\mathsf{P}_1)\cdot f(ra) \subseteq f(\mathsf{P}_1\cdot(ra))\subseteq f(a).\]
This means that $\mathsf{N}(T\cdot f(\mathsf{P}_0)\cdot f(ra))\subseteq \mathsf{N}\cdot f(p(V))\subseteq U$, so we conclude that $\mathsf{M}\cdot f(ra)\denotes$ and
\[\mathsf{M}\cdot f(ra) \subseteq\mathsf{N}(T\cdot f(\mathsf{P}_0)\cdot f(ra))(T\cdot f(\mathsf{P}_1)\cdot f(ra)) \subseteq U\cdot f(a),\]
as desired.
\end{proof}

Now we turn to answering the question posed earlier in this section, under which conditions the functor $\asm(p^\ast,f)$ has a right adjoint.
\begin{thm}\label{cd_is_rightadjoint}
Let $(p,f)\colon (A,\phi)\to (B,\psi)$ be an applicative morphism, and suppose that the underlying category $\mathcal{D}$ of $(B,\psi)$ is a topos.
\begin{itemize}
\item[\textup{(}i\textup{)}]	If $p$ has a right adjoint and $(p,f)$ is computationally dense, then $\asm(p,f)$ has a right adjoint as well.
\item[\textup{(}ii\textup{)}]	If $\asm(p,f)$ has a right adjoint, then $(p,f)$ is computationally dense.
\end{itemize}
\end{thm}

We will first embark on the proof of (i), which is rather involved. The proof of (ii) can be found on page \pageref{proofii}.

\begin{proof}[Proof of \cref{cd_is_rightadjoint}\textup{(}i\textup{)}]
Let $\mathsf{M}\in\psi$ witness the computational density of $(p,f)$, and suppose that $p$ has a right adjoint $q\colon\mathcal{D}\to\mathcal{C}$, where $\mathcal{C}$ is the underlying category of $(A,\phi)$. We denote the unit and counit of the adjunction by $\eta$ and $\varepsilon$ respectively. For the sake of readability, we write $F$ for $\asm(p,f)$. We define its right adjoint $G$.

Let $X$ be an assembly over $(B,\psi)$. First of all, we define
\[E'_X = \{(x,a)\in |X|\times p(A)\mid \forall m,b:B\s (\mathsf{M}(m)\wedge f(a,b)\to mb\denotes\wedge\ E_X(x,mb))\}.\]
Then $q(E'_X)\subseteq q(|X|)\times qp(A)$, and we let
\[E'_{GX} = (\id\times\eta_A)^\ast(q(E'_X)) = \{(x,a)\in q(|X|)\times A\mid q(E'_X)(x,\eta_A(a))\}\subseteq q(|X|)\times A.\]
Now we define the assembly $GX$ over $(A,\phi)$ by setting 
\[|GX| = \{x\in q(|X|)\mid \exists a: A\s (E'_{GX}(x,a))\}\subseteq q(|X|),\]
and by letting $E_{GX}$ be the restriction of $E'_{GX}$ to $|GX|\times A$.

Before we continue, we first formulate the following lemma.
\begin{lem}
For every assembly $X$ over $(B,\psi)$, there is a commutative diagram:
\begin{equation}\label{useful_diagram}
\begin{tikzcd}[column sep=large]
pE'_{GX} \arrow[r] \arrow[d, hook]                 & E'_X \arrow[d, hook] \\
pq|X|\times p(A) \arrow[r, "\varepsilon\times\id"] & {|X|\times p(A)}      
\end{tikzcd}
\end{equation}
\end{lem}
\begin{proof}
The object $E'_{GX}$ is defined by the pullback diagram
\[
\begin{tikzcd}
E'_{GX} \arrow[d, hook] \arrow[r]       & qE'_X \arrow[d, hook] \\
q|X|\times A \arrow[r, "\id\times\eta"] & q|X|\times qpA       
\end{tikzcd}
\]
This means that we can obtain the diagram \eqref{useful_diagram} by pasting the following squares:
\begin{equation}\label{decomp_useful}
\begin{tikzcd}[column sep=large]
pE'_{GX} \arrow[d, hook] \arrow[r]        & pqE'_X \arrow[d, hook] \arrow[r, "\varepsilon"]            & E'_X \arrow[d, hook] \\
pq|X|\times pA \arrow[r, "\id\times p\eta"] & pq|X|\times pqpA \arrow[r, "\varepsilon\times\varepsilon"] & {|X|\times pA}        
\end{tikzcd}
\end{equation}
and using the triangle identity for the bottom composition.
\end{proof}

Now suppose that $g\colon X\to Y$ is an arrow in $\asm(B,\psi)$, tracked by $W\in\psi$. Let $U \in \psi$ be a realizer of $\lambda x.W(\mathsf{M}x)$, and find a $V\in \phi$ such that $V\!A\denotes$ and
\[\mathcal{D}\models \forall r,a: p(A)\s(p(V)(r)\wedge U\cdot f(a)\denotes\ \to \mathsf{M}\cdot f(ra)\subseteq U\cdot f(a)).\]
We claim that
\begin{align}\label{goal}
\mathcal{C}: V(r)\wedge E'_{GX}(x,a)\models_{x:q(|X|); r,a:A} E'_{GY}(q(g)(x),ra).
\end{align}
To this end, we first prove that
\begin{align}\label{step}
\mathcal{D}: p(V)(r)\wedge E'_X(x,a)\models_{x:|X|;r,a:p(A)} E'_Y(g(x),ra).
\end{align}
Reason inside $\mathcal{D}$ and suppose that we have $x\in |X|$ and $r,a\in p(A)$ such that $p(V)(r)$ and $E'_X(x,a)$. If we have $m,b\in B$ such that $\mathsf{M}(m)$ and $f(a,b)$, then $mb\denotes$ and $E_X(x,mb)$. So if $s\in W$, then $s(mb)\denotes$ as well, and $E_Y(g(x),s(mb))$. This means that $U\cdot f(a)\denotes$, and every $c\in U\cdot f(a)$ satisfies $E_Y(g(x),c)$. Now suppose that $m',b'\in B$ such that $\mathsf{M}(m')$ and $f(ra,b')$ are given. By the property of $V$, we know that $m'b'\denotes$ and $m'b'\in U\cdot f(a)$, which implies that $E_Y(g(x),m'b')$. From this, we can conclude that $E'_Y(g(x),ra)$, which proves \eqref{step}.

Now we obtain a commutative diagram
\[\begin{tikzcd}[column sep=large]
pV\times pE'_{GX} \arrow[d, hook] \arrow[r]                                                                     & pV\times E'_X \arrow[d, hook] \arrow[r] & E'_Y \arrow[d, hook] \\
pV\times pq|X|\times pA \arrow[r, "\id\times\varepsilon\times\id"] &                                                            pV\times |X|\times pA \arrow[r, "\ast"]          & {|Y|\times pA}                 
\end{tikzcd}\]
in $\mathcal{D}$, where $\ast$ is the arrow sending $(r,x,a)$ to $(g(x),ra)$. Indeed, the left-hand square exists by diagram \eqref{useful_diagram}, and the right-hand square expresses \eqref{step}. Transposing this diagram yields the diagram
\[\begin{tikzcd}
V\times E'_{GX} \arrow[r] \arrow[d, hook] & qE'_Y \arrow[d, hook] \\
V\times q|X|\times A \arrow[r, "\ast\ast"] & q|Y|\times qpA          
\end{tikzcd}\]
in $\mathcal{C}$, where $\ast\ast$ is the arrow sending $(r,x,a)$ to $(q(g)(x),\eta_A(ra))$. (Observe that, since $V\!A\denotes$, we know that the application map $p(V)\times p(A)\to p(A)$ is the image of the application map $V\times A\to A$ under $p$.) This diagram tells us that
\[\mathcal{C}:V(r)\wedge E'_{GX}(x,a)\models_{x:q(|X|); r,a:A} E'_Y(q(g)(x),\eta_A(ra)).\]
from which \eqref{goal} immediately follows. Since $V$ is inhabited, \eqref{goal} implies that
\[\mathcal{C}: \exists a: A\s (E'_{GX}(x,a))\models_{x:q(|X|)} \exists a: A\s (E'_{GY}(q(g)(x),a)),\]
which means that $q(g)$ restricts to an arrow $G(g)\colon |GX|\to |GY|$. Moreover, \eqref{goal} implies that $V$ tracks $G(g)$ as a morphism $GX\to GY$. It is immediate that $G$ is a functor, so it remains to show that $F\dashv G$.

Since $\mathsf{I}\in\psi$, there exists a $V\in\phi$ such that $V\!A\denotes$ and
\[\mathcal{D}\models \forall r,a: p(A)\s(p(V)(r)\wedge \mathsf{I}\cdot f(a)\denotes\ \to \mathsf{M}\cdot f(ra)\subseteq \mathsf{I}\cdot f(a)),\]
which may be simplified to
\[\mathcal{D}\models \forall r,a: p(A)\s(p(V)(r)\to\mathsf{M}\cdot f(ra)\subseteq f(a)).\]
We claim that for every assembly $X\in\asm(A,\phi)$:
\begin{align}\label{goal2}
\mathcal{D}: p(V)(r) \wedge p(E_X)(x,a)\models_{x:p(|X|); r,a:p(A)} E'_{FX}(x,ra).
\end{align}
Indeed, reason inside $\mathcal{D}$ and suppose that we have $x\in p(|X|)$ and $r,a\in p(A)$ such that $p(V)(r)$ and $p(E_X)(x,a)$. Consider $m,b\in B$ such that $\mathsf{M}(m)$ and $f(ra,b)$. Then by the property of $V$, we know that $mb\denotes$ and $f(a,mb)$. Since $p(E_X)(x,a)$, this implies that $E_{FX}(x,mb)$, so we can conclude that $E'_{FX}(x,ra)$, which proves \eqref{goal2}.

The validity of \eqref{goal2} can be expressed by a diagram
\[\begin{tikzcd}
pV\times pE_X \arrow[r] \arrow[d, hook]  & E'_{FX} \arrow[d, hook] \\
pV\times p|X|\times pA \arrow[r, "\ast"] & p|X|\times pA           
\end{tikzcd}\]
in $\mathcal{D}$, where $\ast$ is the arrow sending $(r,x,a)$ to $(x,ra)$. Transposing this diagram yields a diagram
\[\begin{tikzcd}
V\times E_X \arrow[r] \arrow[d, hook]     & qE'_{FX} \arrow[d, hook] \\
V\times |X|\times A \arrow[r, "\ast\ast"] & qp|X|\times qpA          
\end{tikzcd}\]
in $\mathcal{C}$, where $\ast\ast$ is the arrow sending $(x,r,a)$ to $(\eta_{|X|}(x),\eta_A(ra))$. This diagram implies that
\begin{align}\label{goal22}
\mathcal{C}: V(r) \wedge E_X(x,a)\models_{x:|X|;r,a:A} E'_{GFX}(\eta_{|X|}(x),ra)).
\end{align}
Since $V$ is inhabited and $E_X$ is total, \eqref{goal22} implies that the sequent \[\vdash_{x:|X|}\exists a:A\s(E'_{GFX}(\eta_{|X|}(x),a))\] is valid in $\mathcal{C}$, i.e., the image of $\eta_{|X|}\colon |X|\to qp(|X|) = q(|FX|)$ is contained in $|GFX|$. So we have an arrow $\tilde{\eta}_X\colon |X|\to |GFX|$, and \eqref{goal22} tells us that $V$ tracks it as a morphism $X\to GFX$. The naturality of $\eta$ implies that $\tilde{\eta}$ is natural transformation $\id\To GF$.

Now consider an assembly $X\in\asm(B,\psi)$. We know that $|FGX| = p(|GX|)\subseteq pq(|X|)$, so $\varepsilon_{|X|}\colon pq(|X|)\to |X|$ restricts to an arrow $\tilde{\varepsilon}_X\colon |FGX|\to |X|$. We will show that $\mathsf{M}$ tracks $\tilde{\varepsilon}_X$ as a morphism $FGX\to X$. To this end, reason inside $\mathcal{D}$ and suppose we have $x\in |FGX| = p(|GX|)$ and $m,b\in B$ such that $\mathsf{M}(m)$ and $E_{FGX}(x,b)$. Then there exists an $a\in p(A)$ such that $p(E_{GX})(x,a)$ and $f(a,b)$. This also implies that $p(E'_{GX})(x,a)$ and since we have the diagram \eqref{useful_diagram}, it follows that $E'_X(\varepsilon_{|X|}(x),a)$. Since $f(a,b)$, this means that $mb\denotes$, and $E_X(\varepsilon_{|X|}(x),mb)$, as desired. We conclude that $\tilde{\varepsilon}_X$ is a morphism $FGX\to X$, and that $\tilde{\varepsilon}$ is a natural transformation $FG\To\id$. Moreover, the triangle equalities for $\varepsilon$ and $\eta$ yield that the triangle equalities  hold for $\tilde{\varepsilon}$ and $\tilde{\eta}$ as well, so $F\dashv G$.
\end{proof}

The theory developed above yields a succinct proof of the following result.
\begin{cor}
Let $(A,\phi)$ be a PCA over a topos $\mathcal{C}$. Then $\asm(A,\phi)$ is locally cartesian closed.
\end{cor}
\begin{proof}
Let $f\colon I\to J$ be a morphism of assemblies. As we showed in \cref{pullback=qs}, the applicative morphism $(f^\ast,\delta_{|I|^\ast(A)})\colon (A,\phi)/J\to (A,\phi)/I$ is computationally dense. Moreover, since any topos is locally cartesian closed, we know that the pullback functor $f^\ast\colon\mathcal{C}/|J|\to\mathcal{C}/|I|$ has a right adjoint $\Pi_f$. By \cref{cd_is_rightadjoint}(i), the functor $\asm(f^\ast,\delta_{|I|^\ast(A)})\colon \asm((A,\phi)/J)\to \asm((A,\phi)/I)$ has a right adjoint as well. In \cref{pullback_asm}, we observed that, under the equivalences of \cref{slice}, this functor is naturally isomorphic to the pullback functor $f^\ast\colon \asm(A,\phi)/J\to\asm(A,\phi)/I$. We conclude that this pullback functor always has a right adjoint, i.e., that $\asm(A,\phi)$ is locally cartesian closed.
\end{proof}

Now let us finally also give the proof of \cref{cd_is_rightadjoint}(ii)
\begin{proof}[Proof of \cref{cd_is_rightadjoint}\textup{(}ii\textup{)}]\label{proofii}
Again, let $\mathcal{C}$ be the underlying category of $(A,\phi)$, and write $F$ for $\asm(p,f)$. Suppose we have an adjunction $F\dashv G$ with counit $\varepsilon$. Consider the assembly $S\in\asm(B,\psi)$, where $|S|$ is the object of inhabited subobjects of $B$, and $E_S\subseteq |S|\times B$ is the element relation. Let $\mathsf{N}\in\psi$ be a tracker of $\varepsilon_S\colon FGS\to S$; we claim that $\mathsf{N}$ also witnesses the quasi-surjectivity of $(p,f)$.

Suppose that $U\in\psi$. Then the global section $U\colon 1\to |S|$ is also a morphism $1\to S$, since it is tracked by $\mathsf{K}U$. Since $F1\simeq 1$, this morphism can be transposed to an arrow $\tilde{U}\colon 1\to GS$ of $\asm(A,\phi)$. Then $F\tilde{U}$ is a global section $1\simeq F1 \to FGS$, and by the adjunction, we have $\varepsilon_S(F\tilde{U}) = U$.

Now take $V = \{a\in A\mid E_{GS}(\tilde{U},a)\}$. If $W$ tracks $\tilde{U}$, then $WA$ is defined and a subobject of $V$, which implies that $V\in\phi$. Then $f(p(V)) = \{b\in B\mid E_{FGS}(F\tilde{U},b)\}$, which means that $\mathsf{N}\cdot f(p(V))$ is defined and a subobject of 
\[\{b\in B\mid E_S(\varepsilon_S(F\tilde{U}),b)\} = \{b\in B\mid E_S(U,b)\} = U,\] as desired.
\end{proof}

Putting the above together, we get the following.
\begin{cor}
Suppose that $p\colon \mathcal{D}\to\mathcal{C}$ is a geometric morphism between toposes, and that $(p^\ast,f)\colon (A,\phi)\to(B,\phi)$ is an applicative morphism. Then $\asm(p,f)$ has a right adjoint if and only if $(p,f)$ is computationally dense.
\end{cor}
We can also formulate this result in another way. We may generalize the notion of a geometric morphism to include any adjunction between left exact categories for which the left adjoint is left exact. In this way, the adjunction $\Gamma\dashv \nabla$ becomes a geometric morphism $\mathcal{C}\to\asm(A,\phi)$, which is even an inclusion in the sense that $\Gamma\nabla\simeq\id$.
\begin{cor}
Let $p\colon \mathcal{D}\to\mathcal{C}$ be a geometric morphism between toposes, and let $(A,\phi)$ and $(B,\psi)$ be PCAs over $\mathcal{C}$ and $\mathcal{D}$ respectively. If $(p^\ast,f)\colon (A,\phi)\to(B,\psi)$ is a computationally dense applicative morphism, then there exists an up to isomorpmism commutative diagram
\[\begin{tikzcd}
\mathcal{D} \arrow[r, "p"] \arrow[d, hook] & \mathcal{C} \arrow[d, hook] \\
{\asm(B,\psi)} \arrow[r, "F"]              & {\asm(A,\phi)}             
\end{tikzcd}\]
satisfiying the Beck-Chevalley Condition $F^\ast\nabla \simeq \nabla p^\ast$, with $F^\ast=\asm(p^\ast,f)$. Moreover, every such diagram arises, up to isomorphism, in this way.
\end{cor}

In the case of classical PCAs, it is also known when the geometric morphism induced by a computationally dense applicative morphism is an inclusion (e.g., \cite{jaap},  Proposition 2.6.2).
\begin{prop}
Let $f\colon(A,C)\to (B,D)$ be a computationally dense applicative morphism, and suppose that $\mathsf{m}\in D$ witnesses the computational density of $f$. Then the induced geometric morphism $\asm(A,C)\to\asm(B,D)$ is an inclusion if and only if there exists an $\mathsf{e}\in D$ satisfying:
\[\forall b\in B\s (\mathsf{e}b\denotes\wedge\s \exists a\in A\s (f(a,\mathsf{e}b)\wedge \mathsf{m}\cdot f(a)\denotes\wedge\ \mathsf{m}\cdot f(a)=b)).\]
\end{prop}
The appropriate generalization to PCAs is as follows:
\begin{prop}
Let $(p,f)\colon (A,\phi)\to (B,\psi)$ be a computationally dense applicative morphism, suppose that the underlying category $\mathcal{D}$ of $(B,\psi)$ is a topos, and suppose that $p$ has a right adjoint $q$. If $\mathsf{M}\in \psi$ witnesses the computational density of $(p,f)$, then the induced geometric morphism $\asm(B,\psi)\to\asm(A,\phi)$ is an inclusion if and only if:
\begin{itemize}
\item[\textup{(}a\textup{)}]	$p\dashv q\colon \mathcal{D}\to\mathcal{C}$ is an inclusion, and
\item[\textup{(}b\textup{)}]	there exists an $\mathsf{E}\in\psi$ such that
\[\mathcal{D}\models\forall b,e:B\s (\mathsf{E}(e)\to(eb\denotes\wedge\ \exists a: p(A)\s (f(a,eb)\wedge \mathsf{M}\cdot f(a)=b))),\]
where $\mathsf{M}\cdot f(a)=b$ abbreviates $\forall b',m: B\s (f(a,b')\wedge \mathsf{M}(m)\to mb' = b)$.
\end{itemize}
\end{prop}
\begin{proof}
We write $F=\asm(p,f)$, and we let $G$ be the right adjoint of $F$ as constructed in the proof of \cref{cd_is_rightadjoint}(i). We also write $\eta$ and $\varepsilon$ for the unit resp.\@ counit of $p\dashv q$, and we write $\tilde\eta$ and $\tilde\varepsilon$ for the unit resp.\@ counit of $F\dashv G$ as constructed in the proof of \cref{cd_is_rightadjoint}(i).

First, suppose that $F\dashv G$ is an inclusion. Since we have the commutative diagram
\[\begin{tikzcd}[column sep=large]
\mathcal{D} \arrow[r, "p\s\dashv\s q"] \arrow[d, hook] & \mathcal{C} \arrow[d, hook] \\
{\asm(B,\psi)} \arrow[r, hook, "F\s\dashv\s G"]              & {\asm(A,\phi)}             
\end{tikzcd}\]
of geometric morphisms, this implies that $p\dashv q$ must be an inclusion as well, i.e., $\varepsilon$ is an isomorphism. Now consider the object of realizers $R$ in $\asm(B,\psi)$. In the diagram
\[\begin{tikzcd}
{|FGR|} \arrow[rd, "\tilde\varepsilon_R"'] \arrow[r, phantom, "=" description] & p|GR| \arrow[r, hook] & pqB \arrow[ld, "\varepsilon_B"] \\
                                                                    & B                     &                                
\end{tikzcd}\]
both $\varepsilon_B$ and $\tilde\varepsilon_R$ are isomorphisms, so the inclusion $p|GR|\mono pqB$ is in fact the identity, and $\varepsilon_B = \tilde\varepsilon_R$ (as arrows in $\mathcal{D}$). Now let $\mathsf{E}\in\psi$ be a tracker of $\tilde\varepsilon_R^{-1}\colon R\to FGR$, and reason internally in $\mathcal{D}$. Let $m,e\in B$ and suppose that $\mathsf{E}(e)$. Then $eb\denotes$, and $E_{FGR}(\tilde\varepsilon_R^{-1}(b),eb)$, whence also $E_{FGR}(\varepsilon_B^{-1}(b),eb)$. This means that these exists an $a\in p(A)$ such that $p(E_{GR})(\varepsilon_B^{-1}(b),a)$ and $f(a,eb)$. According to \eqref{useful_diagram}, the former implies that $E'_R(\varepsilon_B(\varepsilon^{-1}_B(b)),a)$, i.e., $E'_R(b,a)$. If $m,b'\in B$ are such that $\mathsf{M}(m)$ and $f(a,b')$, then the definition of $E'_R$ yields that $mb'\denotes$ and $E_R(b,mb')$, i.e., $mb'= b$. This shows that $\mathsf{E}$ has the desired property.

Conversely, suppose that (a) and (b) hold, say that (b) is witnessed by $\mathsf{E}\in \psi$, and let $X$ be an assembly over $(B,\psi)$. First of all, consider the diagram \eqref{decomp_useful}. The left-hand square is a pullback, since it is the image of a pullback diagram under $p$. But since $p\dashv q$ is an inclusion, we know that $p\eta$ is an isomorphism, from which it follows that the arrow $pE'_{GX}\to pq E'_X$ is an isomorphism as well. Since all the $\varepsilon$s are isomorphisms as well, we conclude that the top and bottom arrows in the diagram \eqref{useful_diagram} are isomorphisms. In other words, we have
\[\mathcal{D}\models \forall x:pq(|X|)\s \forall a:p(A)\s (p(E'_{GX})(x,a)\leftrightarrow E'_X(\varepsilon_{|X|}(x),a)).\]
Now reason inside $\mathcal{D}$. Suppose that we have $x\in |X|$ and $b,e\in B$ such that $\mathsf{E}(e)$ and $E_X(x,b)$. Then $eb\denotes$, and there exists an $a\in p(A)$ such that $f(a,eb)$ and $\mathsf{M}\cdot f(a)=b$. Since $E_X(x,b)$, the latter implies that $\mathsf{M}\cdot f(a)\subseteq E_X(x)$, in other words, that $(x,a)\in E'_X(x,a)$. Using the above, this implies that $p(E'_{GX})(\varepsilon^{-1}_{|X|}(x),a)$. So we have shown:
\begin{equation}\label{wat}
\mathcal{D}\models \forall x:|X|\s \forall b,e: B\s (\mathsf{E}(e)\wedge E_X(x,b)\to(eb\denotes\wedge\ \exists a:p(A)\s(f(a,eb)\wedge p(E'_{GX})(\varepsilon^{-1}_{|X|}(x),a))))
\end{equation}
Again, reason inside $\mathcal{D}$ and suppose that $y\in pq(|X|)$. Then $\varepsilon_{|X|}(y)\in |X|$, and since $X$ is an assembly and $\mathsf{E}$ is inhabited, there exist $b,e\in B$ such that $\mathsf{E}(e)$ and $E_X(\varepsilon_{|X|}(y),b)$. Using \eqref{wat}, we find that there is an $a\in p(A)$ such that $p(E'_{GX})(y,a)$. Since $p$ is regular, we have
\[p(|GX|) = \{y\in pq(|X|)\mid \exists a\in p(A)\s (p(E'_{GX})(y,a))\},\]
so we can conclude that $y\in p(|GX|)$ as well. In other words, the inclusion $p(|GX|)\mono pq(|X|)$ is in fact an isomorphism. Modulo this isomorphism, we have that $\tilde\varepsilon_X = \varepsilon_{|X|}$. To finish the proof, we will show that $\mathsf{E}$ tracks $\varepsilon^{-1}_{|X|}\colon |X| \to pq(|X|)\simeq p(|GX|) = |FGX|$ as a morphism $X\to FGX$.

Since
\[\begin{tikzcd}
pE_{GX} \arrow[r, hook] \arrow[d, hook] & pE'_{GX} \arrow[d, hook] \\
p|GX|\times pA \arrow[r, hook]          & pq|X|\times pA          
\end{tikzcd}\]
is a pullback diagram (being the image of a pullback diagram under $p$), we see that the inclusion $p(E_{GX})\mono p(E'_{GX})$ is also an isomorphism. Now reason inside $\mathcal{D}$ and suppose that we have $x\in |X|$ and $b,e\in B$ such that $\mathsf{E}(e)$ and $E_X(x,b)$. Then by \eqref{wat} and the observation we just made, we see that $eb$ is defined, and that there is an $a\in p(A)$ such that $f(a,eb)$ and $p(E_{GX})(\varepsilon^{-1}_{|X|}(x),a)$. But the latter just means that $E_{FGX}(\varepsilon^{-1}_{|X|}(x),eb)$, as desired.
\end{proof}

\end{document}